\documentclass[11pt,a4paper]{amsart}
\usepackage{color}

\usepackage{amsmath,amsfonts,amsthm, amssymb,color}
\usepackage{graphicx}

\newtheorem{theorem}{Theorem} [section]
\newtheorem{lemma}[theorem]{Lemma}
\newtheorem{proposition}[theorem]{Proposition}
\newtheorem{corollary}[theorem]{Corollary}
\newtheorem{conjecture}[theorem]{Conjecture}

\theoremstyle{definition}
\newtheorem{definition}[theorem]{Definition}

\newenvironment{proofof}[1]{\normalsize {\sc Proof of #1}:}{{\hfill $\Box$}}
\newenvironment{mylist}{\begin{list}{}{
\setlength{\parskip}{0mm}
\setlength{\topsep}{1mm}
\setlength{\parsep}{0mm}
\setlength{\itemsep}{0.5mm}
\setlength{\labelwidth}{7mm}
\setlength{\labelsep}{3mm}
\setlength{\itemindent}{0mm}
\setlength{\leftmargin}{12mm}
\setlength{\listparindent}{6mm}
}}{\end{list}}

\DeclareMathOperator{\cycwd}{{\mathbf c}}
\DeclareMathOperator{\decycwd}{{\mathbf d}}

\newcommand{\gen}[1]{\langle#1\rangle}

\newcommand{\wt}[1]{\widetilde{#1}}
\newcommand{\wdl}{\widetilde{L}}
\newcommand{\ldl}{\widetilde{\Lambda}}

\newcommand\R{{\mathbb R}}
\newcommand\N{{\mathbb N}}
\newcommand\Z{{\mathbb Z}}

\newcommand\slex{<_{sl}}
\newcommand\slexeq{\leq_{sl}}
\newcommand\conj{\sim}
\newcommand\subconj{_c}

\newcommand{\pre}{\mathrm{pre}}
\newcommand{\suf}{\mathrm{suf}}
\newcommand{\sub}{\mathrm{sub}}
\def\klt#1{#1-locally testable}

\newcommand\1{\varepsilon}
\newcommand\geol{\mathsf{Geo}}   
\newcommand\ngeol{\mathsf{NearGeo}}   
\newcommand\wgeol{{geodesic language}}   
\newcommand\cyc{\mathsf{Cyc}}    

\newcommand\sphl{\mathsf{SL}}    
\newcommand\wsphl{{shortlex language}}    
\newcommand\wsphs{{spherical growth series}}    

\newcommand\geocl{\mathsf{ConjGeo}}   
\newcommand\wgeocl{{conjugacy geodesic language}}   

\newcommand\geocpl{\mathsf{CycGeo}}   
\newcommand\wgeoccl{{cyclic geodesic language}}  
  
\newcommand\sphcl{\mathsf{ConjSL}}  
\newcommand\wsphcl{{shortlex conjugacy language}}    
\newcommand\sphcs{{\widetilde {\sigma}}}    
\newcommand\wsphcs{{spherical conjugacy growth series}}    

\newcommand\mincl{\mathsf{ConjMinLenSL}}    
\newcommand\wmincl{{shortlex minimal conjugacy language}}    
\newcommand\mineltcl{\mathsf{ConjMinLen}}

\newcommand\wslex{{shortlex}}
\newcommand\Slex{{Shortlex}}
\newcommand\slnf{{shortlex normal form}}
\newcommand\slcnf{{shortlex conjugacy normal form}}

\newcommand\geocon{{conjugacy geodesic}}  
\newcommand\geocp{{cyclic geodesic}} 

\newcommand\ra{\rightarrow}

\begin{document}

\author{Laura Ciobanu, Susan Hermiller, Derek Holt, and Sarah Rees}
\title{Conjugacy languages in groups}
\begin{abstract}
We study the regularity of several languages derived
from conjugacy classes in a finitely generated group $G$
for a variety of examples including
word hyperbolic, virtually abelian, Artin, and Garside groups.
We also determine the rationality of the growth
series of the \wsphcl \ in virtually cyclic groups, proving one 
direction of a conjecture of Rivin.

\bigskip


\end{abstract}
\maketitle

\section{Introduction}\label{sec:intro}

Many classes of finitely presented groups have been
studied via the formal language theoretic
properties of their sets of geodesics 
or shortlex least representatives of the group elements.
In this paper we study the regularity of
three languages derived from the conjugacy classes,
rather than elements, of some of these groups.

A regular language is recognized by a finite state
automaton, and one particularly useful application
is the existence of an algorithm that uses this
automaton to calculate
a rational function equal to the growth series
(that is, the generating function associated to the
growth function) of the language.

Before proceeding further, we need to recall some notation.
We use standard notation from formal language theory
and refer to \cite{hu} for details. Where $X$ is a finite set,
we denote by $X^*$ the set of all words over $X$, and call a subset of
$X^*$ a language.
We write $\1$ for the empty word, 
and denote by $X^+$ the set of all non-empty words over $X$
(so $X^*=X^+ \cup \{\1\}$).
For each word $w \in X^*$, let $l(w)=l_X(w)$ denote its length over $X$.
For subsets $A,B$ of $X^*$, we define $AB$ to be the set of concatenations
$wu$ with $w \in A, u \in B$. Similarly we define $A^n$ to be the set of
concatenations of $n$ words from $A$, $A^* = \cup_{n=0}^\infty A^n$,
$A^+ = \cup _{n=1}^\infty A^n$. A language is regular if it can 
be built out of finite subsets of $X$ using the operations of union,
concatenation, $*$ (and complementation); such an expression for a
language is called a {\em regular expression}. 

All groups we consider in this paper are finitely generated, and all
generating sets finite and inverse-closed.
Let $G=\langle X \rangle$ be a group.
Let $\pi:X^* \rightarrow G$ be the natural projection onto $G$, and let
$=$ denote equality between words and $=_G$ equality between group elements
(so $w=_G v$ means $\pi(w)=\pi(v)$).
For $g \in G$, define the {\em length}
of $g$, denoted $|g|\;(=|g|_X)$, to be
the length of a shortest representative word for $g$ over $X$.
Define a {\em geodesic} to be a word $w \in X^*$ with $l(w)=|\pi(w)|$.

Let $\conj$ denote the equivalence relation on $G$ given
by conjugacy, and $G/\!\conj$ its set of equivalence classes.
Let $[g]\subconj$ denote the conjugacy class of $g \in G$.
Define the {\em length up to conjugacy} of an element $g$ of $G$,
denoted $|g|\subconj$, by 
\[|g|\subconj:=\min\{|h| \mid h \in [g]\subconj\}.\]
We say that $g$ has {\em minimal length up to conjugacy}
if $|g|=|g|\subconj$.

We call a word $w \in X^*$ satisfying $l(w)=|\pi(w)|\subconj$ 
a {\em geodesic}
{\em with respect to conjugacy}, or a {\em \geocon} word.
Note that if a word $w$ is a \geocon, then
so is every cyclic permutation of $w$.
By contrast,
a word $w \in X^*$ satisfying the property
that every cyclic permutation is a geodesic
is not necessarily a \geocon.  Call a word whose
cyclic permutations are all geodesics
a {\em geodesic with respect to cyclic permutation},
or a {\em \geocp} word.

We consider six languages associated to
the pair $(G,X)$, the first three being:  
\begin{eqnarray*}
\geol=\geol(G,X)&:=&\{w \in X^* \mid l(w)=|\pi(w)|\},\\
\geocl=\geocl(G,X)&:=&\{w \in X^* \mid l(w)=|\pi(w)|\subconj\},\\
\geocpl=\geocpl(G,X)&:=&\{w \in X^* \mid w \text{ is a \geocp}\},
\end{eqnarray*}
which we call
the {\em \wgeol}, {\em \wgeocl}, and 
{\em \wgeoccl}, respectively, of $G$ with respect to $X$.
The \wgeocl\ was introduced by the first two authors 
in~\cite{ciobanuhermiller}, where it is shown that
the property of having both $\geol$ and $\geocl$ regular
is preserved by taking graph products.

Although $\geol$ and $\geocl$ capture
much of the geometric information about the elements
and conjugacy classes of $G$,  
it is desirable also
to have languages whose growth functions
are exactly the growth functions of the elements,
minimal length elements up to conjugacy, and 
conjugacy classes of $G$, respectively.
To that end, let $\leq$ be a total ordering of $X$, and 
let $\slexeq$ be 
the induced shortlex ordering of $X^*$ (for which
$u \slex w$ if either $l(u)<l(w)$, or $l(u)=l(w)$ but
$u$ precedes $w$ lexicographically).

For each $g \in G$, 
we define the {\em \slnf} of $g$ to be the unique word
$y_g \in X^*$ with $\pi(y_g)=g$ such that
$y_g \slexeq w$  for all $w \in X^*$
with $\pi(w)=g$.
For each conjugacy class $c \in G/\!\conj$, 
we define the 
{\em \slcnf} of $c$ 
to be the shortlex least word $z_c$ over $X$ representing an 
element of $c$; that is, 
$\pi(z_c)\in c$, and $z_c \slexeq w$ for all 
$w \in X^*$ with $\pi(w)\in c$.

Our remaining three languages are the
{\em \wsphl}, {\em \wmincl}, and 
{\em \wsphcl} for $G$ over $X$, defined respectively as 
\begin{eqnarray*}
\sphl=\sphl(G,X)&:=&\{y_g \mid g \in G\},\\
\mincl=\mincl(G,X)&:=&\{y_g \mid |g|=|g|\subconj\}, \text{ and}\\
\sphcl=\sphcl(G,X)&:=&\{z_c \mid c \in G/\!\conj\}.\end{eqnarray*}
We note that the language $\mincl$ of shortlex
normal forms for the set $\mineltcl$ of minimal length
elements up to conjugacy in $G$ satisfies $\mincl = \geocl \cap \sphl$.

Our six languages satisfy the following containments.
\[\begin{array}{ccccccc} 
       &           & \geocl                      & \subseteq & \geocpl   &
          \subseteq & \geol \\
       &           & \rotatebox{90}{$\subseteq$} &           &           &
           & \rotatebox{90}{$\subseteq$} \\    
         \sphcl & \subseteq & \mincl                      &           &
              \subseteq &           & \sphl 
\end{array}\]
Any language $L$ over $X$ gives rise to a
{\em strict growth function} $\phi_L:\N \cup \{0\} \ra \N \cup \{0\}$, 
defined by  $\phi_L(n) := |\{w \in L \mid l(w) = n\}|$, and
an associated generating function, called
the {\em strict growth series}, given
by $f_L(z) := \sum_{i=0}^\infty \phi_L(i)z^i$.
It is well known 
that if $L$ is a regular language, 
then $f_L$ is a rational function.
For the three languages above, the coefficient
$\phi_{\sphl}(n)$ is the number of elements of $G$  of length $n$,
$\phi_{\mincl}(n)$ is the number of minimal length
  elements of $G$ up to conjugacy of length $n$, and
$\phi_{\sphcl}(n)$ is the number of conjugacy classes of $G$ whose shortest
elements have length $n$.
Note that 
all three of these numbers depend only on $G,X$ and not
on our choice of $\sphl$ as a language of geodesic normal forms for $G$;
the growth series
$f_\sphl$, $f_\mincl$, and $f_\sphcl$
would be the same with respect to any geodesic normal form.
We consider the last of these series in two examples in this paper.
To emphasize this independence from the shortlex ordering, we
denote by $\sphcs$ the strict growth series of $\sphcl$, that is,
\[ \sphcs=\sphcs(G,X) := f_{\sphcl(G,X)}(z)
  = \sum_{i=0}^\infty \phi_{\sphcl(G,X)}(i)z^i, \]
following the notation of \cite{ciobanuhermiller},
and call that series the {\em \wsphcs}.

In Section~\ref{sec:cgvscpg}, 
Proposition~\ref{prop:cyclgeoreggeo}
proves that regularity of $\geol$ implies the same for $\geocpl$.
Theorem~\ref{thm:grprod} proves that equality
of the languages $\geocpl$ and $\geocl$
is preserved by taking graph products,
giving Corollary~\ref{cor:raag}
that $\geocl$ and $\geocpl$ are equal 
(over the standard generating set) for all
right-angled Artin and right-angled Coxeter groups.

Section~\ref{sec:regulargeod} studies regularity of the
conjugacy languages for many families known to have regular geodesic languages.
These include word hyperbolic groups 
(for all generating sets)~\cite[Theorem 3.4.5]{echlpt}
and, with appropriate generating sets, 
virtually abelian groups and geometrically finite hyperbolic groups
~\cite[Theorem 4.3]{neushap},
Coxeter groups~\cite{Howlett},
right-angled Artin groups~\cite{lmw},
Artin groups of large type~\cite{HRgeo}, and
Garside groups (and hence Artin groups of finite type and
torus knot groups)~\cite{CharneyMeier}.
For groups in most of these families, we have succeeded in
proving the regularity of $\geocl$ and $\mincl$.
Moreover, for virtually abelian groups, we 
build on~\cite[Prop.~6.3]{HHR} to exhibit a generating set for
which $\geocl$ has the stronger property of
piecewise testability (see Definition~\ref{def:pw}).

We summarize the results known for these families
of groups in Table~\ref{table2}, where 
the symbol * means that the
property holds for all generating sets, and
if * does not appear, the result is only
known to hold for a specific generating set. 

In the cases of virtually abelian groups and Garside groups, we found no
proofs in the literature of their \wslex\ automaticity with respect to the
generating sets under consideration, and so we have also supplied those
in Section~\ref{sec:regulargeod}.

\begin{table}[!h]
\begin{tabular}
{|@{\hspace{2pt}}c@{\hspace{2pt}}|@{\hspace{3pt}}c@{\hspace{3pt}}|
  @{\hspace{2pt}}c@{\hspace{2pt}}|@{\hspace{2pt}}c@{\hspace{2pt}}|
  @{\hspace{2pt}}c@{\hspace{2pt}}|@{\hspace{2pt}}c@{\hspace{2pt}}|
  @{\hspace{2pt}}c@{\hspace{2pt}}|}
\hline 
   Result & Group  & $\sphl$   & $\geol$  & $\geocl$  & $\mincl$ & $\sphcl$  \\
  \hline \hline
 T\,hm.~\ref{thm:hypgroups} & word hyperbolic & reg.*~\cite{echlpt}&
       reg.*~\cite{echlpt} & reg.* & reg.*    & -- \\ \hline
 Thms.\;\ref{thm:virtcyclic},\ref{thm:hypgroups}  & virtually cyclic
 & reg.*~\cite{echlpt}& reg.*~\cite{echlpt} & reg.*  & reg.* & reg.* \\ \hline  
 Prop.~\ref{prop:virtabcgeo} & virtually abelian & reg. & PT~\cite{HHR}& PT &
       reg. & -- \\ \hline
 Thms.~\ref{thm:Artin_geocl}, & extra-large  & reg.~\cite{HRgeo}&
    reg.~\cite{HRgeo} & reg. & reg. & not reg. \\
    \ref{thm:Artin_sphcl}& type Artin & & & & & \\ \hline
 Prop.~\ref{prop:garsideslaut} & Garside & reg.&reg.~\cite{CharneyMeier}
   &-- &-- & -- \\ \hline
 Thm.~\ref{thm:braid} &homog.~Garside &reg. &reg.~\cite{CharneyMeier} &
    reg. &reg. &-- \\ \hline
   & graph product $\dagger$ & reg.~\cite{hm}& reg.~\cite{lmw} &
      reg.~\cite{ciobanuhermiller} & reg. & --\\ \hline
\end{tabular}

\bigskip
{$\dagger$: assuming that $\sphl$, $\geol$, $\geocl$ are all regular for
all vertex groups}

\bigskip

\vspace{.2cm}
\caption{Summary of language properties for groups in
        Sections \ref{sec:regulargeod} \& \ref{sec:spherical}}\label{table2}
\end{table}
\vspace{-.5cm}
Note that Table~\ref{table2} includes many
families of Artin groups, including those of 
spherical (i.e.~finite) type (which are homogeneous Garside groups),
and also right-angled Artin and Coxeter groups,
as graph products of appropriate vertex groups. 
(The graph product results are included in the table for completeness, but
are all proved elsewhere.)

For two of the blank entries in Table~\ref{table2}
in the $\sphcl$ column, namely word hyperbolic
and graph product groups, we note that
there are examples of groups and finite generating
sets for which $\sphcl$ is not regular for any 
total ordering of the generators.  In particular,
Rivin~\cite{rivin},\cite{rivin04}
and Ciobanu and Hermiller~\cite{ciobanuhermiller}
have shown that for a free product of two infinite cyclic groups 
or two finite cyclic groups of order greater than 2, 
with respect to the cyclic generators (in the infinite case) or the
nontrivial elements of the finite cyclic factors, respectively,
the \wsphcs\ is not rational.

Some of the results in Section~\ref{sec:regulargeod} have consequences for groups beyond those treated in this paper. In \cite[Section 8]{AntolinCiobanu} Antol\'{i}n and the first author have shown that certain relatively hyperbolic groups have regular $\geocl$ by using Corollary \ref{cor:fftp}.

Section~\ref{sec:spherical} is devoted to a 
conjecture of Rivin~\cite{rivin},\,\cite{rivin04}
about the rationality of the \wsphcs\ $\sphcs$ for word hyperbolic
groups. In Theorem~\ref{thm:virtcyclic} we verify one direction of the
conjecture, proving that for a virtually cyclic group $\sphcl$ is regular,
and hence that $\sphcs$ is rational.
This result is included in Table~\ref{table2}.

Section~\ref{sec:cgvsg} examines two virtually abelian groups
to demonstrate how
regularity properties can differ between
the various languages we consider, and depend on choice of generating set.
Cannon~\cite[p.~268]{neushap} showed that regularity
of  $\geol$ depends upon the generating
set for $G=\Z^2 \rtimes \Z/2\Z$.
Using the same two generating sets as Cannon,
we show in Propositions~\ref{prop:sphclnotreg} and~\ref{prop:geovsgeocon} that
regularity of  $\sphcl$  can depend upon the
generating set. At the same time, we prove for one of
these generating sets that, although $\sphcl(G,Z)$ is not regular
(with respect to any total ordering of $Z$),
$\sphcs(G,Z)$ is rational.
Moreover, over the generating set used in 
Proposition~\ref{prop:geovsgeocon}, we show that
$\geocl(G,X)$ is regular but 
$\geol(G,X)$ is not, and similarly $\mincl$
and $\sphcl$ are regular
but $\sphl$ is not.
Considering a further finite extension 
$K=\Z^2 \rtimes D_8$ of $G$ in 
Propositions~\ref{prop:z2d8good} and~\ref{prop:vab_geoconnotreg},
we show that regularity of the \wgeocl\ and \wmincl\ can also
depend upon the generating set.
These results are summarized in Table~\ref{table1}.

\begin{table}[!h]
\begin{tabular}
{|@{\hspace{3pt}}c@{\hspace{3pt}}|@{\hspace{3pt}}c@{\hspace{3pt}}|
  @{\hspace{3pt}}c@{\hspace{3pt}}|@{\hspace{3pt}}c@{\hspace{3pt}}|
  @{\hspace{3pt}}c@{\hspace{3pt}}|@{\hspace{3pt}}c@{\hspace{3pt}}|
  @{\hspace{3pt}}c@{\hspace{3pt}}|}
\hline
   Result                           & Group,                  & $\sphl$      &
    $\geol$                    &  $\geocl$   &  $\mincl$   & $\sphcl$\\
                                    & generators               &              &
                                 &             &             &         \\
  \hline \hline
  Prop.~\ref{prop:sphclnotreg}      & $\Z^2 \rtimes \Z/2\Z,Z$  &  reg.        &
  reg.~\cite{neushap}        & reg.        & reg.        & not reg.    \\ \hline
  Prop.~\ref{prop:geovsgeocon}      & $\Z^2 \rtimes \Z/2\Z,X$  &  not reg. &
 not reg.~\cite{neushap}    & reg.        & reg.        & reg.        \\ \hline
  Prop.~\ref{prop:z2d8good}         & $\Z^2 \rtimes D_8,Z'$    &  reg.        &
 reg.                       & reg.        & reg.        & not reg.    \\ \hline
  Prop.~\ref{prop:vab_geoconnotreg} & $\Z^2 \rtimes D_8,X'$    &  not reg.    &
   not reg.              & not reg.    & not reg.    & not reg.    \\ \hline
\end{tabular}
\bigskip
\caption{Summary of language properties for groups in Section \ref{sec:cgvsg}}\label{table1}
\end{table}

\vspace{-.2cm}

The languages studied in this paper merit exploration not just 
because of the role they play in determining the conjugacy growth 
of a group, but also for their potential connections to the computational 
complexity of the conjugacy problem. 
We observe that a recursively presented group $G= \langle X \mid R \rangle$
for which $\sphcl(G,X)$ is recursive has solvable
conjugacy problem.  (Given $u \in X^*$, we find 
the shortlex representative of $[u]\subconj$ by
simultaneously enumerating conjugates $y^{-1}uy$ of $u$, 
words $w \in \sphcl(G,X)$ with $l(w) \le l(u)$,
and products $z$ of conjugates of relators, and
computing the free reduction of each $y^{-1}uyw^{-1}z$.
When that is the empty word, $w$ is the shortlex representative of
$[u]\subconj$.)
However Chuck Miller has pointed out to us that the
example of an amalgamated product $T$ of two free groups 
of the same rank described in~\cite[Section IVA]{miller}
has unsolvable conjugacy problem 
but has recursive $\geocl$.
Another open question of interest is what regularity of $\sphcl$ or 
$\geocl$ implies for the computational complexity of this important 
decision problem.


\section{Conjugacy versus {\geocp}s}\label{sec:cgvscpg}

For any language $L \subset X^*$, let
$\cyc(L)$ denote the {\em cyclic closure} of $L$;
that is, $\cyc(L)$ is the set of all cyclic 
permutations of words in $L$.  

\begin{lemma}\label{lem:cyclicclosure}
If a language $L$ is regular, then
$\cyc(L)$ is also regular.
\end{lemma}

\begin{proof}
Let $M$ be a finite state automaton accepting $L$,
with state set $Q$ and initial state $q_0$. 
A word $w$ lies in $\cyc(L)$ if and only if,
for some factorisation $w_1w_2$ of $w$,
and for states $q,q' \in Q$, with $q'$ accepting, 
$M$ contains both a path from from $q$ to $q'$ labelled by $w_1$
and a path in $M$ from $q_0$ to $q$ labeled by $w_2$.
From this description we construct a set of non-deterministic
automata $M_q$, indexed by the states of $Q$, the
union of whose languages is $\cyc(L)$.

The automaton $M_q$ is formed from
two disjoint copies $M_{q,1}$ and $M_{q,2}$ of the states and transitions
of $M$. The initial state of $M_q$ is the state $q$ in $M_{q,1}$ and the single
accept state is the state $q$ in $M_{q,2}$. Additional
$\epsilon$-transitions join each state of $M_{q,1}$ that is accepting in $M$
to the state $q_0$ in $M_{q,2}$. 
\end{proof}

\begin{proposition}\label{prop:cyclgeoreggeo}
For $G= \langle X \rangle$,
if $\geol(G,X)$ is regular then so is $\geocpl(G,X)$.
\end{proposition}

\begin{proof}
This follows directly from Lemma~\ref{lem:cyclicclosure}
and the fact that
{\geocp}s are the words
that are not cyclic conjugates of nongeodesics;
that is,
\[
\geocpl =  X^* \setminus \cyc(X^* \setminus \geol).
\]

\vspace{-6pt}
\end{proof}

The following is an immediate consequence of the above proposition.
\begin{proposition}\label{prop:cyclgeo}
For $G = \langle X \rangle$,
if $\geol(G,X)$ is regular and
$\geocpl(G,X) = \geocl(G,X)$, then $\geocl(G,X)$
is also regular.
\end{proposition}

Given a finite simplicial graph $\Lambda$ 
with vertices labeled by groups $G_i$,
the associated graph product group $G$
is generated by the groups $G_i$, with the
added relations that whenever the vertices
labeled $G_i$ and $G_j$ are adjacent,
then the elements of $G_i$ and $G_j$ commute.

\begin{theorem}\label{thm:grprod}
Suppose that 
$G_i=\langle X_i \rangle$, with
$\geocpl(G_i,X_i) = \geocl(G_i,X_i)$, for each $i$.
If $G$ is a
graph product of these groups with generating
set $X=\cup_i X_i$, then $\geocpl(G,X) = \geocl(G,X)$.
\end{theorem}

\begin{proof}
The containment $\geocpl(G,X) \supseteq \geocl(G,X)$
is immediate from the definitions.  

For each index $i$, let 
$\geol_i:= \geol(G_i,X_i)$, and
let $\$$ be a symbol
not in $X_i$.  Define the map 
$\rho_i:X^* \rightarrow (X_i \cup \$)^*$ 
by, for $a \in X_j$, setting
 $\rho_i(a) := a$ if $i=j$,
$\rho_i(a) := \1$ (the empty word) if  the
vertices $i$ and $j$ are adjacent in the
defining graph $\Lambda$ of the graph product,
and finally $\rho_i(a) := \$$ if
the vertices $i$ and $j$ are neither equal
nor adjacent.
In~\cite[Props.~3.3,\,3.5]{ciobanuhermiller},
the first two authors show that, for $w \in X^*$,
\begin{eqnarray*}
&&\!\!\!\!\!\!w \in \geocl(G,X)  \iff \quad\forall i,\\
&&\rho_i(w) \in \geocl(G_i,X_i)  \cup 
\{u_0 \$ u_1 \cdots \$ u_n \mid n \ge 1
\text{ and } u_nu_0,u_1,...,u_{n-1} \in \geol_i\};
\end{eqnarray*}
moreover, $w$ is geodesic if and only if
$\rho_i(w) \in \geol_i (\$ \geol_i)^*$ for all $i$.

Suppose now that
$w\in\geocpl(G,X)$.
For each index $i$, we have 
$\rho_i(w) = u_0 \$ u_1 \cdots \$ u_n$ with
$n \ge 0$ and each $u_j \in \geol_i$.
Note that whenever $w'$ is a cyclic permutation
of $w$, then $\rho_i(w')$ is a cyclic
permutation of $w'$.  In particular there
is a cyclic permutation $w'$ of $w$ such
that $\rho_i(w') = u_nu_0 \$ u_1 \cdots u_{n-1} \$$.
Now since $w' \in \geol(G,X)$ as well, we have
$u_nu_0,u_1,\ldots,u_{n-1} \in \geol_i$.
Suppose further that $n=0$, that is, 
$\rho_i(w)=u_0 \in \geol_i$.  Now for every
cyclic permutation $u'$ of $u_0$, there is
a cyclic permutation $w'$ of $w$ such that
$\rho_i(w')=u'$, and hence we must have
$u' \in \geol_i$.  Thus in this case 
$u_0 \in \geocpl(G_i,X_i)$,
and so by hypothesis also $u_0 \in \geocl(G_i,X_i)$.
Therefore $w \in\geocl(G,X)$,
as required.
\end{proof}

In~\cite{ciobanuhermiller} the first two authors show that 
whenever the languages $\geol,\geocl$ are both regular for
all vertex groups, then both are regular
for the graph product group.  Theorem~\ref{thm:grprod},
however, gives a slightly stronger result about the
structure of the geodesics involved, which we highlight
in the case of right-angled Artin and Coxeter groups.

\begin{corollary}\label{cor:raag}
For every right-angled Artin group 
(respectively, every right-angled Coxeter group)
$G$ with respect to
the Artin (respectively, Coxeter) generators $X$,
a word $w$ is a \geocon\ if and only if
every cyclic conjugate of $w$ is geodesic;
that is, $\geocl(G,X) = \geocpl(G,X)$.
\end{corollary}


\section{Results about groups with regular geodesic languages}
  \label{sec:regulargeod}

In this section we study conjugacy languages associated to particular
families of groups, including word hyperbolic, virtually abelian, locally
testable, (extra) large type Artin, and Garside groups,
for which the geodesic language is regular.  In each of these cases, we prove
that the \wgeocl\ $\geocl$ is also regular.


\subsection{Word hyperbolic groups}\label{subsec:hyperbolic}

For a word hyperbolic group $G$, regularity of the
geodesic and shortlex languages
holds for every finite generating set~\cite[Thms. 3.4.5,\,2.5.1]{echlpt}.
In the following we show that the same is true for the
\wgeocl.

\begin{theorem}\label{thm:hypgroups}
Let $G=\langle X \rangle$ be a word hyperbolic group. Then $\geocl(G,X)$ 
and $\mincl(G,X)$ are regular.
\end{theorem}

\begin{proof}
Since $\mincl(G,X)=\geocl(G,X) \cap \sphl(G,X)$, and
$\sphl$ is regular, it suffices to prove the result for $\geocl$.
We suppose that $G$ has hyperbolicity constant $\delta \geq 0$ (i.e.
geodesic triangles in the Cayley graph are $\delta$-slim).
By \cite[Lemma III.2.9]{BridsonHaefliger}, two words $u,v \in \geocpl(G,X)$
with $\max\{l(u), l(v)\} \geq 8\delta +1$ represent conjugate elements precisely
when there exist $u' \in \cyc(\{u\}), v' \in \cyc(\{v\})$ representing elements
conjugate by a word of length at most $2\delta+1$. 
From this we deduce that a word $v \in \geocpl$ of length at least $8\delta+1$
is outside $\geocl$ precisely when it is in the cyclic closure of the set
$ \bigcup_{|\alpha| \leq 2\delta + 1} L(\alpha) $
where, for $\alpha \in X^*$,
\[ L(\alpha)  := \{ v' \in \geocpl\mid  \exists u' \in
\geocpl \text{ such that } 
\alpha^{-1}u'\alpha =_G v',\  l(v')>l(u'))\}. \]

Since the set $\geocl \cap (\cup_{k=0}^{8\delta+1} X^k)$ is finite,
it is sufficent to prove the regularity of the set
\[ \cyc\left(\bigcup_{|\alpha| \leq 2\delta + 1} L(\alpha)\right).\]
Since the family of regular sets is closed under both finite union
and cyclic closure (Proposition~\ref{lem:cyclicclosure}), it is enough
to show the regularity of $L(\alpha)$, for a given word $\alpha$.

Now, since $G$ is hyperbolic it has a biautomatic structure on $\geol$. 
By \cite[Lemma~8.1]{gs}, for any word $\alpha$, the language
\[ L_1(\alpha) := \{ (u,v)\mid  u, v \in \geol(G,X),\ 
v=_G \alpha^{-1}u\alpha\}\]
is regular.

By Proposition~\ref{prop:cyclgeoreggeo}, $\geocpl(G,X)$ is regular and
by standard arguments so are the languages 
\begin{eqnarray*}
L_2&:=&\{(u,v) \mid u,v \in \geocpl(G,X),\  l(v) > l(u) \},\\
L_3(\alpha) &:=& L_1(\alpha) \cap L_2, \hspace{.5in} \text{and}\\
L_4(\alpha) &:=&
\{ v \in \geocpl\mid \exists u \in \geocpl,\ (u,v) \in L_3(\alpha) \}.
\end{eqnarray*}
We see that $L_4(\alpha)=L(\alpha).$
\end{proof}

In contrast, for the \wsphcl,
Rivin~\cite{rivin},\,\cite{rivin04}
and Ciobanu and Hermiller~\cite{ciobanuhermiller}
have given 
examples of word hyperbolic groups
and specific generating sets for which the  
the \wsphcs\ is not rational,
and so $\sphcl$ is not regular for any 
total ordering of the generators.
More precisely, $\sphcs$ was proved non-rational for
free products of the form $\Z* \Z$ and $\Z/m\Z * \Z/n\Z$ with $m,n>2$,
using as generators in the first case cyclic generators for the two copies
of $\Z$ and in the second case all nontrivial elements of the two finite
cyclic groups.


\subsection{Virtually abelian groups}\label{subsec:virtuallyabelian}

In~\cite[Props.~4.1,\,4.4]{neushap} Neumann and Shapiro show that
for every virtually abelian group, there is a generating set
for the group such that the corresponding set of geodesic words
is regular, and in~\cite[Prop.~6.3]{HHR}, Hermiller, Holt, and Rees
strengthen this result by showing that the geodesic words are
a piecewise testable language.
In the following we adjust that proof to show that
for some choice of generators
the language of conjugacy geodesics is also piecewise testable.

\begin{definition}\label{def:pw}
Let $A$ be a finite alphabet.

(1) A subset $L$ of $A^*$ is called \textit{piecewise testable} if it is
defined by a regular expression that combines terms of the form $A^*a_1A^*a_2
\dots A^* a_k A^*$ using the Boolean operations of union, intersection and
complementation, where $k\geq 0$ and each $a_i \in A$.

(2) A subset $L$ of $A^*$ is called \textit{piecewise excluding} 
if there is a finite set of strings $W \subset A^*$ with the property 
that a word $w \in A^*$ lies in $L$ if and only if $w$ does not 
contain any of the strings in $W$ as a not necessarily consecutive 
substring. In other words,
$$L=( \cup_{i=1}^n \{A^* a_{i_1} A^*a_{i_2} A^* \dots A^* a_{i_{l_i}}A^*\})^c,$$
where $W=\{a_1, \dots, a_n\}$ and $a_i=a_{i_1} a_{i_2} \dots a_{i_{l_i}}$ 
for $1 \leq i \leq n$.
\end{definition}

According to Definition~\ref{def:pw}, piecewise excluding 
languages are also piecewise testable.

\begin{proposition}\label{prop:virtabcgeo}
Let $G$ be a virtually abelian group. There exists a finite generating set 
$Z$ for $G$ such that $\geocl(G,Z)$ is piecewise testable. 
Furthermore, there is an ordering of $Z$ with respect to which $G$ is
\wslex\ automatic and $\mincl(G,Z)$ is regular.
\end{proposition}
\begin{proof}
This argument follows closely the proof of \cite[Prop.~6.3]{HHR}, 
so we omit some of the details.

Let $G$ be a virtually abelian group, and let
$N \lhd G$ be abelian of 
finite index in $G$ with a finite generating set $A$. 
Let $T \cup \{1\}$ be a transversal of $N$ in $G$.

We build the generating set $Z$ for $G$ as follows.
Let $Y := T^{\pm 1}$.
Let $X'$ be the set of all $x \in N$ such that
$x =_G w \neq_G 1$ for some 
$w \in Y^*$ with $l(w) \le 4$.
Finally, let $X$ be the closure of the set $A \cup X'$ in $G$ under
inversion and conjugation, and let 
$Z := X \cup Y$.  Then 
\begin{mylist}
\item[(i)] $X \subset N$, $Y \subset G \setminus N$,
\item[(ii)] both $X$ and $Y$ are closed under inversion,
\item[(iii)] $X$ is closed under conjugation by elements of $G$, 
\item[(iv)] $Y$ contains at least one representative of each nontrivial coset of $N$
in $G$, and 
\item[(v)] if $w =_G xy$ with $w \in Y^*$, $l(w) \le 3$, 
$x \in N$ and $y \in Y \cup \{\1\}$, then $x \in X$.
\end{mylist}
Write the finite set $X$ as $X=\{x_1,...,x_m\}$.
For each $x \in X$ and $y \in Y$, let $x^y$ denote
the generator in $X$ that represents the group element $y^{-1}xy$.
Similarly if $v=x_{i_1} \cdots x_{i_k} \in X^*$, 
the symbol $v^y$ denotes the word $x_{i_1}^y \cdots x_{i_k}^y$.
An immediate consequence of  properties (i)-(v) above is that
the language $L:=\geol(G,Z)$ of geodesics of $G$ over 
$Z$ satisfies the property that
$L \subseteq X^{\star} \cup X^{\star} Y  X^{\star} \cup X^{\star} Y
X^{\star}Y  X^{\star}.$

We can now establish the \wslex\ automaticity of $G$ with respect to a
suitable (total) ordering of $Z$. Choose any such ordering in which the
generators in $X$ precede those in $Y$.
Note that every element of $G$ has a geodesic representative in the
set $X^{\star} \cup X^{\star} Y$ and since, by properties (iii) and (v),
elements of the form $yz$ with $y \in Y$ and $z \in Z$ have an
alternative representative
of the form $x'y'$ with $x' \in X \cup \{\1\}$ and 
$y' \in Y \cup \{\1\}$, we see that
all \wslex\ minimal representatives have this form.
Let $L'$ be the set of words $w$ or $wy$ of this form in which
$w \in X^*$ is \wslex\ minimal.
By \cite[Thm. 4.3.1]{echlpt}, a finitely generated abelian group is
\wslex\ automatic with respect to any finite ordered generating set.
It follows (again from properties (iii) and (v)) that $L'$ is the language
of an automatic structure for $G$. The closure of $Y$ under inversion
means that elements of $G$ may have two representatives in $L'$, but by
\cite[Thm. 2.5.1]{echlpt}, we get the \wslex\ automatic structure for $G$ by
intersecting $L'$ with $\sphl(G,Z)$.   

Turning now to the set $\wdl:=\geocl(G,Z)$ of conjugacy geodesics
of $G$ over $Z$, we have $\wdl:=\geocl(G,Z) \subseteq L$,
and so $\wdl$ can be partitioned as the union of the subsets
$\wdl_0:=\wdl \cap X^*$, $\wdl_1 := \wdl \cap X^*YX^*$, and
$\wdl_2:=\wdl \cap X^*YX^*YX^*$.
We show that each $\wdl_i$ is a
piecewise testable language.

We start by showing that $\wdl_0=L\cap X^*$. 
Suppose that $w\in L\cap X^*$ is not a conjugacy geodesic,
and write $w=a_1 a_2 \dots a_n$ with $a_i \in X$. 
Then there exist $n \in N$ and 
$y\in Y$ such that the element $(ny)^{-1}wny$ of $G$ is represented
by a word $u$ that is shorter than $w$. Since $N$ is 
abelian we have $u =_G (ny)^{-1}wny =_G y^{-1}wy$. 
Then $w=_G u^{y^{-1}}$ and by properties (ii)-(iii) 
the formal conjugate $u^{y^{-1}}$ is a word in $X^*$
satisfying 
$l(u^{y^{-1}}) = l(u)<l(w)$, which contradicts the fact 
that $w$ is geodesic. By the proof of~\cite[Prop.~6.3]{HHR}, 
$L\cap X^*$ is piecewise excluding, and hence also 
piecewise testable, so $\wdl_{0}$ has the same properties.

An operation on words over $Z$ given
by replacement $ayxb \rightarrow ax^{y^{-1}}yb$
with $a,b \in Z^*$, $x \in X$, and $y \in Y$
is called a {\em $Y$-shuffle}.
An operation on words over $X$ given by a replacement
$ux_ix_jv \rightarrow ux_jx_iv$ is a {\em shuffle}.
Note that whenever a word $w$ can be obtained from a word $v$
by means of finitely many applications of these operations,
then $v =_G w$, $v$ and $w$ are words over $Z$ of
the same length, and furthermore $v \in \wdl_i$ if
and only if $w \in \wdl_i$.

Next we turn to $\wdl_1$. 
In this case we further partition the set 
$\wdl_1 = \cup_{r \in Y} \wdl_{1,r}$
where $\wdl_{1,r}:=\{v_1rv_2 \in \wdl \mid v_1,v_2 \in X^*\}$
for each $r$ in $Y$.
Let $\ldl_{1,r} := \{w \in X^* \mid wr \in \wdl\}$. It is immediate to see that $\wdl_{1,r}$ 
is the set of all words that can
be obtained from words of the form 
$wr$ with $w \in \ldl_{1,r}$ using $Y$-shuffles.

Next let
$\wt{U}_{1,r} := \ldl_{1,r} \cap x_1^* x_2^* \cdots x_m^*$,
and let
\[U_{1,r} := \{(n_1,...,n_m) \in \N_0^m \mid
x_1^{n_1} \cdots x_m^{n_m} \notin \wt{U}_{1,r}\}.\] 
Now $U_{1,r}$ has only finitely many elements that are minimal under the
ordering on $\N_0^m$ defined
by  $(n_1,...,n_m) \le (p_1,...,p_m)$ if and only if
$n_i \le p_i$ for all $1 \le i \le m$
(for a proof see, for example, \cite[Lemma 4.3.2]{echlpt}); 
let $S_{1,r}$ be this finite set of minimal elements.
Let $\wt{S}_{1,r}$ be the set of all words that
can be obtained from elements of the set
$\{x_1^{n_1} \cdots x_m^{n_m} \mid (n_1,...,n_m) \in S_{1,r}\}$
using shuffles, and 
let $\wt{S}_{1,r}'$ be the set of all words
that can be obtained from elements of
the set $\{wr \mid w \in \wt{S}_{1,r}\}$
via $Y$-shuffles.  Then the set 
$\wdl_{1,r}$ of $Y$-shuffles
of words in $\ldl_{1,r} r$
is exactly the set of
all words in $X^*$ that do not contain a piecewise
subword lying in the finite set $\wt{S}_{1,r}'$.
So $\wdl_{1,r}$ is piecewise excluding,
and hence also piecewise testable, for 
each $r \in Y$.  Therefore $\wdl_1$ is also piecewise
testable.  

The proof that $\wdl_2$ is piecewise testable starts by partitioning 
$\wdl_2 = \cup_{r,s \in Y} \wdl_{2,r,s}$
where $\wdl_{2,r,s}:=\{v_1rv_2sv_3 \in \wdl \mid v_1,v_2,v_3 \in X^*\}$
for each $r,s$ in $Y$.
The proof is similar to that for $\wdl_1$, but
a little more complicated, and we omit the details.

Having established that $\geocl$ and $\sphl$ are regular,
it follows that their intersection $\mincl$ is also regular.
\end{proof}

An immediate consequence of the proof above
is that for abelian groups
piecewise testability of the \wgeocl\ 
holds for every generating set.

\begin{corollary}\label{cor:fgab}
If $G=\langle X \rangle$ is abelian,
then $\geocl(G,X)$ is piecewise testable.
\end{corollary}


\subsection{Groups with locally testable geodesics}\label{sec:ltgeod}

This section is devoted to those groups, previously studied in
\cite{HHRloctest}, for which
$\geol(G,X)$ is locally testable.  Free groups, free abelian groups
and dihedral Artin groups over
the standard generators provide examples, as do direct
products of such groups.

Informally, whenever $k$ is a positive integer, a language $L$ is
\klt{$k$} if membership of a word in $L$ depends
on the nature of its subwords of length $k$
(where by a subword of 
a word $a_1a_2\cdots a_n$, we mean either the
empty word or a contiguous substring  $a_i a_{i+1}\cdots a_j$ for some
$1 \le i \le j \le n$); a language $L$ is locally testable if it is \klt{$k$} for some $k$.

More precisely, $k$-local testability is defined
as follows.
Let $k > 0$ be a natural number.
For $u \in X^*$ of length at least $k$, let $\pre_k(u) $ be the prefix
of $u$ of length $k$, let $\suf_k(u) $ be the suffix
of $u$ of length $k$, and let $\sub_k(u) $ be the set
of all subwords of $u$ of length $k$.  If $l(u)<k$,
then we define $\pre_k(u)=u$, $\suf_k(u)=u$, and 
$\sub_k(u)=\emptyset$.
We define an equivalence relation $\sim_k$ on $X^*$ to relate $u$ and $v$
when 
$\pre_{k-1}(u)=\pre_{k-1}(v)$, $\suf_{k-1}(u)=\suf_{k-1}(v)$, 
and $\sub_{k}(u)=\sub_{k}(v)$.
A subset $L \subseteq X^*$
is defined to be {\it \klt{$k$}} \cite[p.~247]{BrzozowskiSimon}
if $L$ is a union of 
equivalence classes of $\sim_k$. 
We refer the reader to
\cite{HHRloctest} and the references cited there
for additional information.
Here we show the following.

\begin{theorem}\label{thm:lt}
For a group $G = \langle X \rangle$, if
$\geol(G,X)$ is locally testable, then 
\linebreak
$\geocl(G,X) = \geocpl(G,X) \setminus A$, where $A$ is a finite set.
Hence  $\geocl(G,X)$ is regular.
\end{theorem}

\begin{proof}
The key ingredient of this proof is \cite[Lemma 5.2]{HHRloctest}, which can be
reformulated as follows: if $G  = \langle X \rangle$ and
$\geol(G,X)$ is locally testable, then there exists $N\in \N$
such that, for each word $w \in X^*$ with $l(w)>N$, there is a cyclic
permutation $\widetilde{w}$ of $w$ with the property that, if $\widetilde{w}$
is a geodesic, then $\widetilde{w}^j$ is a geodesic for all $j \geq 1$.

Assume that $\geol(G,X)$ is locally testable.
We claim that each $w \in \geocpl(G,X)$ with $l(w)> N$ is in fact in
$\geocl(G,X)$. Notice first that the cyclic
geodesic $w$ lies in $\geocl(G,X)$ if and only if each of
its cyclic permutations is in $\geocl(G,X)$.
Now let $\widetilde{w}$ be the cyclic permutation of $w$ provided 
by \cite[Lemma 5.2]{HHRloctest}. 
Suppose that $\widetilde{w}\not\in \geocl(G,X)$.
Then there exist $u, v \in X^*$ such that $\widetilde{w} =_G u ^{-1}v u$ with
$l(v) < l(\widetilde{w})$. Notice that 
$l(\widetilde{w}^j)=j\ l(\widetilde{w})$ for
all $j \geq 1$.
Then
\[|u^{-1} v^j u|\leq 2\,l(u) +l(v^j) = 2\,l(u) +j\,l(v) \leq
2\,l(u)+(j\,l(\widetilde{w})-j),\]
which for $j> 2\,l(u)$ leads to $|\tilde w^j|=|u^{-1} v^j u| <
j\,l(\widetilde{w})=l(\widetilde{w}^j)$,
contradicting the fact that $\widetilde{w}^j$ is a geodesic.
\end{proof}

\subsection{Groups with the falsification by fellow traveler property}
For a group $G = \langle X \rangle$ and $k \ge 0$,
we say that the words $w,w' \in X^*$ {\em $k$-fellow travel}
(and write $w\asymp_k w'$),  
if, for each $i \geq 0$, $|\pre_i(w)^{-1}\pre_i(w')| \leq k$.
We say that $G = \langle X \rangle$ satisfies the
{\em falsification by fellow traveler property} (FFTP) if,
for some fixed constant $k$, any nongeodesic word $w$ $k$-fellow
travels with a shorter word. We will write $k$-FFTP if the above
constant needs to be mentioned explicitly.
In~\cite[Prop.~4.1]{neushap} Neumann and Shapiro showed that
a group $G$ satisfying the FFTP over a generating set $X$ 
has $\geol(G,X)$ regular.
They also show that the FFTP is dependent upon
the generating set. 
In order to apply this property to conjugacy languages, we begin
with a strengthening of the result of \cite[Prop.~4.1]{neushap}.

\begin{proposition}\label{prop:FFTPngeol}
Suppose that $G=\langle X \rangle$ satisfies the FFTP.
Then for every $K\geq 0$ the set 
\[ \ngeol_K := \{ w \in X^*: |w| \geq l(w)-K\} \]
is regular.
\end{proposition}
\begin{proof}
Suppose that $w \not\in \ngeol_K$. So $|w|\leq l(w)-(K+1)$.
Suppose that $G$ satisfies the $k$-FFTP. Then there is a sequence of words
$w_0=w,w_1,\ldots,w_N$, all representing the same element of $G$ as $w$
where, for each $i$, $w_i\asymp_k w_{i-1}$, $l(w_i)<l(w_{i-1})$, and where
$l(w_N)<l(w)-K$. So, assuming that $N$ is minimal with $l(w_N)<l(w)-K$,
we have  $N\leq K+1$, and hence $w\asymp_{(K+1)k} w_N$.  So 
\[ X^* \setminus \ngeol_K = \{ w: \exists w', 
l(w')<l(w)-K, w \asymp_{(K+1)k} w', w=_G w' \}.\]
This is the projection onto the first coordinate of the
intersection of the padded languages 
$L_1 := \{(w,w') \mid l(w')<l(w)-K\}$ and 
$L_2 :=\{((w,w') \mid w \asymp_{(K+1)k} w', w=_G w' \}$;
see \cite[Section~1.4]{echlpt} for details on languages
of padded pairs.
The regularity of $X^* \setminus \ngeol_K$, and hence also of 
$\ngeol_K$, now follows from the regularity of $L_1$ (which
we leave to the reader) and of $L_2$. 
The set $L_2$ is the regular language accepted by 
the ``standard automaton''
$M_\epsilon$ of \cite[Definition~2.3.3]{echlpt},
associated to the identity element of $G$ and
based on $(W,N)$, where $N$ is the ball of radius $(K+1)k$
centered at the identity in $G$, and $W$ is a finite state
automaton accepting the language $X^*$.
Briefly, given an automaton $M$ with alphabet $X \cup \{\$\}$,
state set $S$,
initial state $s_0$ and transition function 
$\delta:S \times X \cup \{\$\} \ra S$ accepting the
regular language $L(W)\$^*=X^*\$^*$,
then $M_\epsilon$ has alphabet 
$(X \cup \$)^2 \setminus \{(\$,\$)\}$,
state set $S \times S \times N$, and 
initial state  
$(s_0,s_0,\epsilon)$.  When $M_\epsilon$
is in state $(s,t,g)$ and reads a letter
$(a,b)$, the automaton goes to the
state $(\delta(s,a),\delta(s,b),a^{-1}gb)$ if
$a^{-1}gb \in N$, and fails otherwise;
a word is accepted by $M_\epsilon$ if
the automaton finishes in a state of the
form $(s,t,\epsilon)$ where $s$ and $t$
are accept states of $M$.
\end{proof}

\begin{proposition}
\label{prop:FFTPgeocpl}
Suppose that $G=\langle X \rangle$ satisfies the FFTP.
Suppose that, for some fixed $k$,
$L$ is defined to be the set of words $w \in \geocpl(G,X)$
for which the following condition holds:
\[ (v \in \cyc(w)\wedge \alpha \in X^* \wedge l(\alpha)\leq k)
  \Rightarrow |\alpha^{-1}v\alpha| \geq |v|.\]
Then $L$ is regular.
\end{proposition}

\begin{proof}
For a word $w \in \geocpl$, we have
$w\not\in L$ if and only if
$w \in \cyc( \cup_{\alpha \in  X^*,l(\alpha)\leq k} L(\alpha))$,
where \[ L(\alpha) = \{ v \in \geocpl:  |\alpha^{-1}v\alpha|<|v|\}.\]
Since any word $v$ satisfies
\[|\alpha^{-1}v\alpha|< |v| \iff
l(\alpha^{-1}v\alpha) - |\alpha^{-1}v\alpha| > 2l(\alpha),\]
we have
\[ L(\alpha)= \{ v \in \geocpl: \alpha^{-1}v\alpha \in \ngeol_{2l(\alpha)}\}. \] 
Since $G$ satisfies the FFTP, the languages $\ngeol_{2l(\alpha)}$ and
$\geol = \ngeol_0$ are regular by 
Proposition~\ref{prop:FFTPngeol},
and hence (by Proposition~\ref{prop:cyclgeoreggeo})
so is $\geocpl$.
Standard properties of regular languages now ensure the regularity
of $L(\alpha)$.

Since the class of regular languages is closed under finite union, and
by cyclic closure (Lemma~\ref{lem:cyclicclosure}),
the complement of $L$ in $\geocpl$ given by
$ \cyc( \cup_{\alpha \in X^*,l(\alpha)\leq k} L(\alpha))$ is also regular.
It follows that $L$ is a regular set.
\end{proof}

The following is immediate.
\begin{corollary}\label{cor:fftp}
Suppose that $G=\langle X \rangle$ satisfies the FFTP.
Suppose that, for some fixed $k,s$, and for all $r \geq s$,
the set  $\geocl(G,X) \cap X^r$ is equal to the set of all
$w \in \geocpl(G,X) \cap X^r$ for which
\[ (v \in \cyc(w)\wedge \alpha \in X^* \wedge l(\alpha)\leq k)
                    \Rightarrow |\alpha^{-1}v\alpha| \geq |v|.\]
Then $\geocl(G,X)$ is regular.
\end{corollary}

Applying this corollary with
$s=8\delta+1$ and $k=2\delta+1$ gives an alternative proof of 
Theorem~\ref{thm:hypgroups}
that uses FFTP rather than biautomaticity in word hyperbolic groups. 

In the following section we use Proposition~\ref{prop:FFTPgeocpl}
(rather than the corollary)
to prove the regularity of $\geocl$ for an extra-large type Artin group.


\subsection{Artin groups of (extra) large type}\label{sec:artin}

An Artin group is defined by the presentation
\[ \langle x_1,\ldots ,x_n \mid {}_{m_{ij}}(x_i,x_j)=
{}_{m_{ij}}(x_j,x_i)\quad\hbox{\rm for each}\quad i \neq j \rangle, \]
where $(m_{ij})$ is a {\em Coxeter matrix}
(a symmetric $n \times n$ matrix with entries in
$\N \cup \{\infty\}$, $m_{ii}=1, m_{ij} \geq 2$, $\forall
i \neq j$),
and where
for generators $a,a'$ and $m \in \N$ we define ${}_m(a,a')$ to be the
word that is the product of $m$ alternating $a$s and $a'$s that
starts with $a$.
The set $\{x_1,\ldots.x_n\}$ is usually called the standard generating set
of the group, but since we want our generating sets to be inverse-closed,
we define $X := \{x_1,\ldots.x_n\}^{\pm 1}$ to be the {\em standard generating
set}.
An Artin group has {\em large type} if all the integers $m_{ij}$ are
at least 3, and {\em extra-large type} if they are all at least 4.

Holt and Rees~\cite[Thms.~3.2,\,4.1]{HRgeo} have shown that the 
languages $\geol(G,X)$ and $\sphl(G,X)$ are regular.
Theorems~\ref{thm:Artin_geocl} and
~\ref{thm:Artin_sphcl} below
investigate $\geocl(G,X)$, $\mincl(G,X)$, and $\sphcl(G,X)$
for Artin groups of (extra-)large type.

First we note that it follows from a result of Mairesse and
Math\'eus~\cite[Prop.~4.3]{MM} that
the set of geodesics for a dihedral Artin group
(i.e. an Artin group presented above with $n=2$)
over its standard generating set is locally testable, and
so we obtain the following corollary of Theorem~\ref{thm:lt}.

\begin{corollary}
\label{cor:dihedral_Artin_geocl}
For a dihedral Artin group $G$ over its standard
generating set $X$, the set $\geocl(G,X)$ is regular.
\end{corollary}

\begin{theorem}
\label{thm:Artin_geocl}
Let $G=\langle X \rangle$ be an Artin group of extra-large type with $X$ its
standard generating set. Then $\geocl(G,X)$ is regular, and
for any ordering of the generators $\mincl(G,X)$ is also regular.
\end{theorem}

\begin{proof}
Let $n:= |X|/2$.
For standard generators $x_i,x_j$, we write $X(i,j) = \{ x_i,x_j\}^{\pm 1}$
and $G(i,j) = \langle X(i,j) \rangle$.
When $n=2$, the Artin group is dihedral, and so the result follows from
Corollary~\ref{cor:dihedral_Artin_geocl}.

When $n \geq 3$, $\geocl \setminus \{1\}$ can be written as a union
$L_1 \cup L_2 \cup L_3$
of words involving the generators $x_i^{\pm 1}$ for one, two or 
at least three
values of $i$, respectively.
By \cite[Prop. 4.1]{HRconj}, $L_1$ is exactly the set of
freely reduced
powers of generators, 
and hence is regular.
We aim now to show that $\geocl(G(i,j),X(i,j)) \subseteq \geocl(G,X)$.
So suppose that $w$ is a 2-generator word involving generators $x_i,x_j$,
and that $w \not\in \geocl(G,X)$.
If $w$ is non-geodesic in $G$, then, by results of~\cite[Section~3]{HRgeo},  
$w \not \in \geocl(G(i,j),X(i,j))$, so from on now we assume that $w$ is
geodesic in $G$.
Suppose that some generator $g$ conjugates $w$ to a word with a shorter
representative. The results proved in \cite{HRgeo} show that none of the
reductions used to reduce a words to \wslex\ normal form could involve
$g$ if $g \not\in X(i,j)$, so $g \in X(i,j)$, and 
again $w \not \in \geocl(G(i,j),X(i,j))$.
Otherwise, the element $\pi(w)$ is `cyclically reduced' according
to the definition of \cite{HRconj}; that is, for each
$a \in X$ we have $|a^{-1}\pi(w)a| \ge |\pi(w)|$.
Then we can apply
\cite[Prop. 5.1]{HRconj}, which ensures the existence of words
$\alpha,u$ over $X(i,j)$ with
$\pi(\alpha^{-1}w\alpha)=\pi(u)$ and $l(u)<l(w)$.
So in this case too $w \not\in \geocl(G(i,j),X(i,j))$.

Since $\geocl(G(i,j),X(i,j)) \supseteq (L_1 \cup L_2) \cap X(i,j)^*$, 
it now follows that
\[ L_1 \cup L_2 = \cup_{i\neq j} \geocl(G(i,j),X(i,j)). \]
Consequently $L_1 \cup L_2$ is regular. 

We shall now apply Proposition~\ref{prop:FFTPgeocpl} to show that $L_3$
is regular.
We note that since $\geocl(G,X) \subseteq \geocpl(G,X)$, we have
$L_3 \subseteq \geocpl(G,X)$.  By~\cite[Thm.~4.1]{HRgeo},
the language $\geol(G,X)$ is regular, and so  
Proposition~\ref{prop:cyclgeoreggeo} shows that $\geocpl(G,X)$
also is regular.  Standard properties of regular languages now show
that the set $\tilde L_3$ of all words
in $\geocpl$ involving at least three generators is also regular. 

Words in $\geocpl$ are `specially cyclically reduced'
according to the criteria of~\cite[Thm. $4''$]{AppelSchupp}, 
and so that theorem
applies to show that two words in $\tilde L_3$ represent conjugate
elements if and only if cyclic conjugates of them are conjugate via
a power of a generator. Then \cite[Prop. 6.2]{HRconj}
applies to show that those two cyclic conjugates must have
the same length unless one of them represents an element of
$G$ that is not `cyclically reduced'
according to the definition above; that is,
its conjugate by some generator represents a shorter element of $G$.
Also, by \cite[Props. 4.2,\,5.1]{HRconj}, a word in $\tilde L_3$
cannot be conjugate to a shorter element of $G$ that involves fewer than
three generators.

Hence, for $w \in \tilde L_3$,
$w$ lies in $\geocl(G,X)$ precisely when none of its cyclic
conjugates is shortened by conjugation by a generator; that is
precisely when
\[ (v \in \cyc(w)\wedge a \in X ) \Rightarrow |a^{-1}va| \geq |v|.\]
In~\cite[Thm.~4.1]{HRgeo} it is also shown that the group $G$ satisfies the
FFTP with respect to the generating set $X$.
Now Proposition \ref{prop:FFTPgeocpl} applies to show that the set of all words
$v$ that lie in $\geocpl$ and satisfy the above condition is regular.
But we have just shown that $L_3$ is the intersection of this set with the
regular set $\tilde L_3$, so $L_3$ is regular,
and the proof that $\geocl(G,X)$ is regular is complete.

Holt and Rees~\cite[Thm.~3.2]{HRgeo} show that $\sphl(G,X)$
is regular for any ordering of the generators,
and therefore the intersection
$\mincl(G,X)=\geocl(G,X) \cap \sphl(G,X)$ is also regular.
\end{proof}

The remainder of this section is devoted to the proof of
Theorem~\ref{thm:Artin_sphcl}.
We need a technical result about a subset of $\geocl$ in the case where $G$ is 
dihedral.

\begin{lemma}
\label{DA_technical}
Let $G=\langle a,b \mid {}_m(a,b) = {}_m(b,a) \rangle$
be a dihedral Artin group with $m \ge 3$, and let $\Delta = {}_m(a,b)$.
Let $w$ be a positive word in $a^2$ and $b^2$.
Suppose that $w'$ is any word in the generators and inverses that represents a
conjugate of $\pi(w)$ in $G$.  Then $|w'| \ge |w|$ and, if $|w'|=|w|$, then $w'$ is
a cyclic conjugate of $w$ or of $w^\Delta$.

\end{lemma}
In order to make the proof of the lemma easier to follow, we precede it with a
short explanation of the structure of geodesics in dihedral Artin groups;
a characterisation is given in \cite{MM}.

Following \cite{HRgeo}, we extend the notation for alternating products of
generators already
introduced above, and define, for any letters (not just generators) $a,b$,
${}_m(b,a)$ and $(b,a)_m$ to be words of length consisting of
alternating $a$s and $b$s that begin and end in $a$, respectively.
So (for example)  ${}_m(a,b)^{-1} = (b^{-1},a^{-1})_m$.

Now, for any word $w$ 
in the generators (and inverses)  of the dihedral Artin group
$G$ presented as in the statement of the lemma,
we define $p(w)$ to be the minimum of $m$ and the length of the longest
subword of $w$ of alternating $a$'s and $b$'s.  Similarly, we define $n(w)$ to
be the minimum of $m$ and the length of the longest subword of $w$ of
alternating $a^{-1}$'s and $b^{-1}$'s.
According to \cite[Prop.~4.3]{MM}, $w$ is geodesic if and
only if $p(w)+n(w) \le m$.
Furthermore, $\pi(w)$ has more than one geodesic representative if and only
if $p(w)+n(w)=m$.

We recall also that the element $\Delta$ represented by 
${}_m(a,b)$ and ${}_m(b,a)$ is central
in $G$ when $m$ is even, whereas $a^\Delta=b$, $b^\Delta=a$ and $\Delta^2$
is central when $m$ is odd.
For a word $w = a_1\cdots a_k$ with $a_i \in \{a,b\}^{\pm}$, we define
$w^\Delta$ to be the word $a_1^\Delta\cdots a_k^\Delta$.

\begin{proofof}{Lemma~\ref{DA_technical}}

We prove the result using a minimal counterexample argument,
and for the induction to work we need to work under the more general
hypothesis that $w$ is a cyclic conjugate of a positive word in $a^2$
and $b^2$. So either $w$ is itself
such a word, or else it has the form $x\hat{w}x$, where $\hat{w}$ is a
positive word in $a^2$ and $b^2$, and $x \in \{a,b\}$.
For such a word $w$, $p(w) \leq 2$, and $n(w)=0$. So, by \cite{MM},
as explained above, $w$ is geodesic.

Suppose that $g^{-1}wg =_G w'$, and let $v$ be a shortest word representing $g$.
Choose $w,w'$ and $g$ to be a counterexample to the theorem in which $|v|$
is minimal.

Then $p(v) \le m/2$, since otherwise there would be a counterexample with
shorter $v$ in which a positive alternating subword $u$ of $v$ of length $p(v)$
is replaced by the geodesic representative
of $\Delta^{-1}u$, the prefix $v_1$ of $v$ before
the occurrence of $u$ is replaced by $v_1^\Delta$,
and $w'$ is replaced by $w'^\Delta$. Similarly
$n(v) \le m/2$. We assume that $v$ has a positive prefix, that is,
that the first letter of $v$ is $a$ or $b$.
(Otherwise the last letter of $v^{-1}$ is $a$ or $b$, and the argument is
similar.) In fact, we may assume without loss that it is $a$.

If there is any free cancellation in the word $v^{-1}wv$, then we can replace
$w$ by a cyclic conjugate and $v$ by a shorter word so, by the minimality
of $v$, the word $v^{-1}wv$ must be freely reduced, and it cannot be a geodesic word.
Since any negative alternating subwords of $v^{-1}wv$ must occur within
$v$ or $v^{-1}$, we have $n(v^{-1}wv) \le m/2$. Since $v^{-1}wv$ is
not geodesic, we see (by considering the cases when $m$ is even and odd)
that $p(v^{-1}wv) > (m+1)/2$. Hence, since $p(v) \le m/2$ and
$p(w) \le 2 \le (m+1)/2$, the longest positive alternating subword of
$v^{-1}wv$ must overlap the subwords $w$ and $v$ (and hence, since $v$
begins with $a$, $w$ must end in $b$).

Suppose first that $m$ is odd. 
In this case, for any $0<p<m$ the equation
${}_p(a,b)=_G \Delta_{m-p}(a^{-1},b^{-1})$ holds, and we use it below.
Now, since $p(v) \le (m-1)/2$ and
$p(v^{-1}wv) > (m+1)/2$, $w$ must end in $a^2b$ and $v$ must have a prefix
$u := {}_{(m-1)/2}(a,b)$. Let $w = w_1ab$, where $w_1$ ends in $a$.  Then
\begin{eqnarray*}
 u^{-1}wu  &=& (b^{-1},a^{-1})_{(m-1)/2}\, w_1\,{}_{(m+3)/2}(a,b)\\&=_G&
 (b,a)_{(m+1)/2}\, w_1^\Delta\, {}_{(m-3)/2}(a^{-1},b^{-1})\\&=&
  (b,a)_{(m-3)/2}\, (abw_1)^\Delta\, {}_{(m-3)/2}(a^{-1},b^{-1}).
\end{eqnarray*}
So, by replacing $w$ by the cyclic conjugate $(abw_1)^\Delta$ of $w^\Delta$,
and the prefix $u$ of $v$ by the shorter word
${}_{(m-3)/2}(a^{-1},b^{-1})$, we find a counterexample with a shorter $v$.

So now suppose that $m$ is even.
In this case, for any $0<p<m$ the equation
${}_p(a,b)=_G \Delta_{m-p}(b^{-1},a^{-1})$ holds, and $\Delta$ is central.
Suppose first that $w$ does not have $ab$ as suffix. Then $v$ must have a
prefix $u := {}_{m/2}(a,b)$. Let $w = w_2b$.  Then 
\begin{eqnarray*}
 u^{-1}wu  &=& (b^{-1},a^{-1})_{m/2}\, w_2\,{}_{m/2+1}(b,a)\\&=_G&
 (a,b)_{m/2}\, w_2^\Delta\, {}_{m/2-1}(a^{-1},b^{-1})\\&=_G&
 (a,b)_{m/2}\, w_2\, {}_{m/2-1}(a^{-1},b^{-1})\\&=&
  (b,a)_{m/2-1}\, bw_2\, {}_{m/2-1}(a^{-1},b^{-1}).
\end{eqnarray*}
So, by replacing $w$ by its cyclic conjugate $bw_2$, and the prefix $u$ of $v$
by the shorter word ${}_{m/2-1}(a^{-1},b^{-1})$, we find a counterexample
with a shorter $v$.

So now suppose that $w$ does have $ab$ as suffix, and let $w = w_3ab$, where $w_3$
ends in $a$.  If $v$ has prefix $u := {}_{m/2}(a,b)$ then
\begin{eqnarray*}
 u^{-1}wu  &=& (b^{-1},a^{-1})_{m/2}\, w_3\,{}_{m/2+2}(a,b)\\&=_G&
 (a,b)_{m/2}\, w_3^\Delta\, {}_{m/2-2}(b^{-1},a^{-1})\\&=&
 (a,b)_{m/2}\, w_3\, {}_{m/2-2}(b^{-1},a^{-1})\\&=&
  (a,b)_{m/2-2}\, abw_3\, {}_{m/2-2}(a^{-1},b^{-1}).
\end{eqnarray*}
So, by replacing $w$ by its cyclic conjugate $abw_3$, and the prefix $u$ of $v$
by the shorter word ${}_{m/2-2}(a^{-1},b^{-1})$, we find a counterexample
with a shorter $v$.

Finally, if $v$ does not have prefix ${}_{m/2}(a,b)$, then it must have
prefix ${}_{m/2-1}(a,b)$, and either $v$ or $v^{-1}$ must contain a
negative alternating subword of length $m/2$. In that case, we can
replace $g^{-1}wg$ by an equivalent word of length $|g^{-1}wg| -2$, in
which a negative alternating subword of $v$ or $v^{-1}$ is replaced by
a positive alternating subword of the same length, and the subword
${}_{m/2+1}(a,b)$ of $g^{-1}wg$ that overlaps $w$ and $g$ is replaced by
${}_{m/2-1}(b^{-1},a^{-1})$. The resulting word has no positive or negative
alternating subwords of length greater than $m/2$, and hence it is
geodesic. Since $|v|>1$, the resulting word is longer than $w$, and 
so $w$, $g$, and $v$ do not give
a counterexample to the theorem, which is a contradiction of the
original choice of these words.
\end{proofof}

\begin{theorem}
\label{thm:Artin_sphcl}
Let $G$ be an Artin group of large type, and $X = \{x_1,\ldots,x_n\}$
its standard generating set. 
Then $\sphcl(G,X)$ is not regular.
\end{theorem}
\begin{proof}
Suppose first that $n=2$, so $G=\langle a,b \mid {}_m(a,b) = {}_m(b,a) \rangle$
(with $a=x_1$, $b=x_2$) is a dihedral Artin group with $m \ge 3$.
Assume, without loss of generality, that $a<b$. 
We deduce from Lemma~\ref{DA_technical} that, for $m,n>2$,
the word $a^{2m}b^2a^{2n}b^2$ is in $\sphcl$ if and only if $m\geq n$.
Hence
$\sphcl(G,X) \cap (a^2)^*b^2(a^2)^*b^2 = \{ a^{2m}b^2a^{2n}b^2: m \geq n \}$,
which is not regular.
It follows from the regularity of $(a^2)^*b^2(a^2)^*b^2$
that $\sphcl(G,X)$ is not regular.

When $n > 2$, we let $a=x_i$ and $b=x_j$ be any two generators in $X$.
Now, as we remarked in the proof of Theorem~\ref{thm:Artin_geocl},
\cite[Prop. 5.1]{HRconj} implies that, if a group element
$a^{2m}b^2a^{2n}b^2$ with $m,n>2$ is conjugate in $G$ to a shorter word, then
it is conjugate to a shorter word in the subgroup
$G(i,j) = \langle a,b \rangle$ of $G$. (In fact the proof of
\cite[Prop. 5.1]{HRconj} is valid in general only for extra-large
type Artin groups, but it works for large type provided that the
element $g'$ defined in that proof has at least three syllables, which
is true for the 4-syllable word $a^{2m}b^2a^{2n}b^2$.) So again, assuming
$a<b$, we have 
$\sphcl(G,X) \cap (a^2)^*b^2(a^2)^*b^2 = \{ a^{2m}b^2a^{2n}b^2: m \geq n \}$,
which is not regular, and 
hence $\sphcl(G,X)$ is not regular.
\end{proof}


\subsection{Garside groups}\label{sec:garside}
{\em Garside groups} (also known as {\em small Gaussian groups})
were introduced in \cite{DehornoyParis99} as a generalization
of the spherical type Artin groups, which include the braid groups.
It is shown there that many of the
properties of  spherical type Artin groups, including having a geodesic
biautomatic structure, generalize to Garside groups. A Garside group $G$
is defined by a finite presentation in which all of the relations are of the
form $v=u$, where $v$ and $u$ are positive words in the group generators.  So
this presentation also defines the associated {\em Garside monoid} $G^+$, and
it turns out that this embeds into the Garside group.

There is an element $\Delta \in G^+$ known as the {\em Garside element}, and
the (typically rather large) set $S$ of divisors of $\Delta$ in $G^+$ forms the
set of {\em Garside generators} of $G$. The geodesic biautomatic structure
mentioned above is defined on these generators, and we shall show now that this
is also a \wslex\ automatic structure.

\begin{proposition}\label{prop:garsideslaut}
Let $G$ be a Garside group with Garside element $\Delta$ and Garside generators
$S$. Then there is an ordering of $X := S \cup S^{-1}$ with respect to which
$G$ is \wslex\ automatic. 
\end{proposition}

\begin{proof}
We first recall some standard notation and results on Garside groups.
For $a,b,c \in G^+$, with $c =_{G^+} ab$, we say that $a$ is a {\em left
divisor} and $b$ is a {\em right divisor} of $c$. It can be shown that
any two elements $a,b \in G^+$ have a unique `largest' common left divisor
$a \wedge b$, which is left-divisible by all of their common left divisors.
For $a \in G^+$, we can write 
$a=_{G^+} (a \wedge \Delta)a'$ for some
$a' \in G^+$, and the word 
$a_1a_2a_3\cdots a_k \in S^*$ where
the elements $\hat a_1:=a$ and
$\hat a_{i+1}:=\hat a'_i$ for $1 < i < k$ satisfy
$\hat a'_k = 1$ and
$a_{i}:= \hat a_i \wedge \Delta$ for $1 \le i \le k$,
is known as the {\em left greedy normal form} of $a$.

Let $L$ be the set of words over $X$ of the form $u^{-1}v$, where
$u, v \in S^*$, $u$ and $v$ are in left greedy normal form, and
$u \wedge v = 1$.  It is proved in \cite[Thm. 8.1]{DehornoyParis99}
that $L$ is the language of a geodesic biautomatic structure for $G$
with uniqueness. We shall now define an order of $X$ with respect to which the
words in $L$ are the \wslex\ least representatives of the group elements.

We choose any total ordering of $X$ with the following three properties.
\begin{mylist}
\item[(i)] $s^{-1} < t$ for all $s,t \in S$;
\item[(ii)] if $s,t \in S$ and $s$ is a left divisor of $t$, then $t<s$.
\item[(iii)] if $s,t \in S$ and $s$ is a right divisor of $t$, then
$s^{-1}<t^{-1}$.
\end{mylist}
Let $w$ be the \wslex\ least representative of the group element with normal
form $u^{-1}v \in L$. We claim that $w = u^{-1}v$. Since
$u^{-1}v$ is geodesic, we have $|w|=|u^{-1}v|$. Since the normal forms
of elements of $G^+$ lie in $S^+$, we see that the least $m \ge 0$ with
$\Delta^m w \in G^+$ is equal to $|u|$ and the least $n \ge 0$ such that
$\Delta^n w^{-1} \in G^+$ is equal to $|v|$. It follows that $w$ must
contain exactly $|u|$ generators from $S^{-1}$ and $|v|$ from $S$
because if, for example, it contained $m < |u|$ negative generators,
then we would have $\Delta^m w \in G^+$.
So, by property (i) of the ordering, we have $w = u'^{-1}v'$ with
$u',v' \in S^*$, $|u'|=|u|$, $|v'|=|v|$.
So $u' =_G u$ and $v' =_G v$  by \cite[Cor. 7.5]{DehornoyParis99}.

It follows immediately from the definition of the left greedy normal
form and from property (ii) of the ordering that $v$ is the \wslex\ 
least representative of its group element, and hence $v=v'$.
Let $u' = a_1'a_2' \cdots a_m'$ with $a_i' \in S$.
By \cite[Lemma 8.4]{DehornoyParis99}, if $u'$ is not in normal
form then $a_i'a_{i+1}'$ is not in normal form for some $i$.
So, if the normal form word for this element is $a_ia_{i+1}$, then
$a_i'$ is a left divisor of $a_i$, and hence $a_{i+1}$ is a right
divisor of $a_{i+1}'$. But then, by (iii), $a_{i+1}^{-1}a_i^{-1} <
a_{i+1}'^{-1}a_i'$, contradicting the \wslex\ minimality of $u'^{-1}$.
So $u'$ is in normal form and hence $u'=u$, and $w=u^{-1}v$ as claimed.
\end{proof}

Let $G$ be a Garside group with Garside generators $S$ and $\Delta \in S$.
Conjugation by $\Delta$ permutes $S$, and we define
$\tau:S \rightarrow S$ by
$\tau(x) = \Delta^{-1} x \Delta$ for $x \in S$.
The {\em left greedy normal form} of an element $g \in G$
(which is a different normal form from the one in the proof of
Proposition \ref{prop:garsideslaut}, agreeing with
that normal form only on elements of $G^+$) 
has the form
$\Delta^p a_1a_2 \cdots a_k$, where $a_1 a_2 \cdots a_k$ is a
positive word in left greedy normal form with $a_i \ne \Delta$
(see, for example,~\cite[2.6]{BirmanKoLee} for details). So
$p$ is maximal such that $\Delta^p$ is a left divisor of $g$.
Define $\inf(g) := p$, and $\sup(g):=p+k$.

The {\em cycling} and {\em decycling} operations from $G$ to $G$ are defined by
\begin{eqnarray*}
\cycwd(g) &=& \Delta^p a_2 \cdots a_k \tau^{-p}(a_1),\\
\decycwd(g) &=& \Delta^p \tau^p(a_k)a_1a_2 \cdots a_{k-1}.
\end{eqnarray*}
Note that these operations are equivalent to conjugation by
$\tau^{-p}(a_1)$ and $a_k^{-1}$, respectively. The (de)cycled word is not
necessarily in normal form, and it has to be put into normal form before the
operation can be applied again.  Note that cycling and decycling do
not decrease $\inf(g)$, but they may increase it. Similarly they do not
increase $\sup(g)$, but they may decrease it. 

Let $\inf([g])$ and $\sup([g])$ denote respectively the largest value of
$\inf(h)$ and the smallest value of $\sup(h)$ for an element $h$ in the
conjugacy class of $g$.

In \cite[Thm. 1]{BirmanKoLee}, it is proved that, for the braid group
$B_n$, there is a fixed number $K$ (equal to $(n^2-n)/2-1$ or $n-2$, depending
on which set of Garside generators of $B_n$ is used) such that
\begin{enumerate}
\item if $\inf([g]) > \inf(g)$ then $\inf(\cycwd^k(g)) > \inf(g)$
for some $k \le K$;
\item if $\sup([g]) < \sup(g)$ then $\sup(\decycwd^k(g)) < \sup(g)$
for some $k \le K$.
\end{enumerate}

A Garside group is called {\em homogeneous} if $l(v)=l(u)$ for each of its
defining relations $u=v$. This implies that the positive words in the group
generators that represent an element $g \in G^+$ all have the same length.
The spherical type Artin groups are homogeneous Garside groups, but there are
examples of inhomogeneous Garside groups.

It is observed in \cite[Section 3.2]{Franco} that the proof of
\cite[Thm. 1]{BirmanKoLee}
works for all homogeneous Garside groups, but it relies heavily
on elements of $G^+$ having a well-defined length in $G^+$, so it does
not appear to extend to general Garside groups.
It is proved in \cite[Props.~3.7,\,3.10]{picantin} 
that \cite[Thm. 1]{BirmanKoLee} holds
for general Garside groups, but without the bound on $k$.
For the remainder of this section, we assume that $G$ is homogeneous.

As observed above, $G$ is automatic with a geodesic normal form on
$X := S \cup S^{-1}$.  It is pointed out in
\cite[proof of Corollary 3]{BirmanKoLee} that the
geodesic length $|g|$ of $g$ over $X$ is equal to
$\max(\sup(g),\sup(g)-\inf(g),-\inf(g))$;
in fact $|g|$ is equal to $\sup(g)$ when $g$ is positive, $-\inf(g)$
when $g$ is negative, and $\sup(g)-\inf(g)$ otherwise.  

So, if $g$ is conjugate in $G$ to a shorter element $h$, then at least
one of $\inf(h) > \inf(g)$ and $\sup(h) < \sup(g)$ is true.
In the first case, we are in case (1) above, 
and so for some $k \leq K$, we have 
$\inf(\cycwd^k(g)) > \inf(g)$. We can deduce that $|\cycwd^k(g)|<|g|$ except
possibly when $g$ is positive and $|g|=\sup(g)$. But in this case, we must
have $\sup(h) < |g|=\sup(g)$, and so we are also in case (2), and can deduce
that 
$|\decycwd^{k'}(g)|=\sup(\decycwd^{k'}(g)) < \sup(g)=|g|$
for some $k' \le K$.
In the second case, we prove the identical result analogously. 

Hence, since $\cycwd^k$ and $\decycwd^k$ correspond to conjugation by group
elements of length $k$, we have the following.

\begin{proposition}\label{prop:garsidebound}
Let $G$ be a homogeneous Garside group with Garside
generators $S$. If  $w \in (S \cup S^{-1})^*$, and there exists $h \in G$ with
$|h^{-1}wh| < |w|$, then there exists such an $h$ of length at most $K$.
\end{proposition}

It is proved in \cite{Holt10} that $G$ satisfies the FFTP over the Garside
generators, and so Propositions~\ref{prop:garsidebound}
and~\ref{prop:garsideslaut}
together with Corollary \ref{cor:fftp} implies the following.

\begin{theorem}\label{thm:braid}
If $G$ is a homogeneous Garside group with Garside generators $S$
and $X := S \cup S^{-1}$, then $\geocl(G,X)$ is regular, and
there is an ordering of $X$ for which $\mincl(G,X)$ is regular.
This holds, in particular, for spherical type Artin groups.
\end{theorem}


\section{\Slex\ languages and \wsphs}\label{sec:spherical}

In \cite{rivin} and \cite{rivin04}, Rivin proved that
the \wsphcs\ $\sphcs(G)$ is not rational for nonabelian free groups on
their free generating sets, and made the following conjecture, 
one direction of which we shall prove in this section.

\begin{conjecture}\cite[Conjecture 13.1]{rivin}\label{conj:rivin}
Let G be a word hyperbolic group. Then
$\sphcs(G)$ is rational if and only if G is virtually cyclic.
\end{conjecture}

\begin{theorem}\label{thm:virtcyclic}
Let $G$ be a virtually cyclic group.
Then for all generating sets of $G$ the set of shortlex conjugacy
normal forms $\sphcl$ is regular and 
hence the spherical conjugacy growth series $\sphcs$ is rational.
\end{theorem}

\begin{proof}
We may assume that $G$ is infinite. Then
there exists $H \unlhd G$, $H = \gen{x} \cong \Z$, with $G/H$ is finite.
Let $C:=C_G(H)$ be the centralizer of $H$ in $G$.
Then the conjugation action of $G$ on $H$ defines a map
$G \rightarrow {\rm Aut}(\Z)$ with kernel $C$ and so $|G : C| \leq 2$.
For $g \in G \setminus C$, we have $x^{-1}gx=gx^2$, and hence the coset $Hg$
is either a single conjugacy class in $\langle H,g \rangle$ or a union of two
such classes, containing $g$ and $gx$. So $G \setminus C$ consists of
finitely many conjugacy classes of $G$.
On the other hand, for $g \in C$, $|G:C_G(g)|$ is finite, so
$C$ is a union of infinitely many finite classes. 

Since $\sphcl \cap (G \setminus C)$ is finite, it is regular,
and to prove regularity of  $\sphcl$ it is
enough to show that $\sphcl \cap C$ is regular.

Let $T$ be a transversal of $H$ in $G$. Then for each $c \in C$, the
conjugacy class of $c$ is $\{t^{-1}ct \mid t \in T\}$,
and hence any word $w$ with $\pi(w)=c$ is in $\sphcl$ if and only if
there does not exist $t \in T$  for which $t^{-1}wt$ has a
representative $v$ with $v<_{sl} w$.

Now $G$ is word hyperbolic, and we recall from the proof Theorem
\ref{thm:hypgroups} that the set
\[ L_1(t) := \{ (u,v): u, v \in \geol,\quad \pi(v)= \pi(t^{-1}ut)\}\]
is regular for any $t \in T$, as is the set $\geol$.
So $\sphcl \cap C$ is the intersection of $\pi^{-1}(C)$ with
\[ \geol \setminus \cup_{t \in T} (\{ u \in \geol:
    \exists v \in \geol \text{ such that } (u,v) \in L_1(t),\, v <_{sl} u\}).\]
Now standard arguments for regular sets show that this set regular, and
$|G:C|$ finite implies that $\pi^{-1}(C)$ is regular,
so $\sphcl \cap C$ is also regular.
\end{proof}

%
%
\section{Behavior of conjugacy languages for virtually abelian groups}\label{sec:cgvsg}

In this section we present two examples of virtually abelian groups 
and discuss their conjugacy languages. 
Our main goal is to demonstrate that
regularity of the conjugacy languages can
occur even when $\sphl$ and $\geol$, or others among
the conjugacy languages, are not regular, and
to show dependence of regularity 
of the three conjugacy languages on the 
generating sets of the groups.
However, we recall that we proved in 
Proposition~\ref{prop:virtabcgeo} 
that any virtually abelian group has some generating set with respect
to which both $\geol$ and $\geocl$ are regular
and an ordering on that generating set for which
both $\sphl$ and $\mincl$ are regular.

We begin by considering an index 2 extension $G$ of $\Z^2$, where
$G$ is the semidirect product 
$\Z^2 \rtimes \Z/2\Z$, and the $\Z/2\Z$ action
swaps the generators of $\Z^2$; that is, 
$$
G=\langle a,b,t \mid t^2=1, ab=ba, a^t=b \rangle.
$$
Cannon~\cite[p.~268]{neushap} noted that 
the language of geodesics for $G$
can be either regular 
or non-regular, 
depending on the generating set.  In the following
two propositions we explore Cannon's generating sets
in the case of conjugacy languages.  

\begin{proposition}\label{prop:sphclnotreg}
Let
$G \cong \Z^2 \rtimes \Z/2\Z$ be as defined above, 
with the generating set
$Z=\{a^{\pm 1},b^{\pm 1},t\}$.
Then $\sphcl(G,Z)$ is not regular with 
respect to any ordering on $Z$, but 
the \wsphcs\ is given by the rational function
\[
\sphcs = \frac{(1+z)(1+2z+3z^2-z^3-z^4)}{(1-z^2)^2}.
\]
Moreover $\geocl(G,Z)$ and $\mincl(G,Z)$ are regular.
\end{proposition}

\begin{proof}
Note that, since $t=t^{-1}$, the generating set
$Z$ is inverse-closed.
With respect to any ordering on $Z$ with $a^{\pm 1}<b^{\pm 1}<t$,
the \wsphl\  is
\[
\sphl(G,Z) = \{a^ib^jt^\epsilon \mid 
i,j \in \Z, \epsilon \in \{0,1\}\}.
\]
Using this representation of the group elements, 
we compute the conjugacy classes as
\[
[a^ib^j]\subconj=\{a^ib^j,a^jb^i\},\quad\quad
[a^ib^jt]\subconj = 
\{a^{i+k}b^{j-k}t \mid k \in \mathbb{Z} \}.
\]

Thus, using this ordering, the \wsphcl\ is the non-regular set
\begin{eqnarray*}
\sphcl(G,Z)&=&
\{a^i b^j \mid |i| > |j|, i,j \in \mathbb{Z} \} 
\cup 
\{a^ib^i \mid i \in \Z\}
\cup
\{a^ib^{-i} \mid i \in \N\}\\
&& \quad \cup \{a^i t \mid i \in \Z\}, \end{eqnarray*}
For any ordering of $Z$, the
intersection of $\sphcl(G,Z)$ with the regular set 
$\{a,b\}^*$ is either
$\{a^i b^j \mid i>j \ge 0 \}$ or $\{b^i a^j \mid i>j \ge 0 \}$,
and so $\sphcl(G,Z)$ is never regular.
However,
the above expression for $\sphcl$ gives
\[\sphcs = 1+3z+\sum_{k=1}^\infty (4k+3)z^{2k} 
          + \sum_{k=1}^\infty (4k+4)z^{2k+1},\]
yielding the required rational function for $\sphcs$.

Finally, applying the structure of the conjugacy classes 
again together with the \wgeol\ 
$
\geol(G,Z) = \bigcup_{\zeta,\eta \in \{\pm 1\}} 
\{a^\eta,b^{\zeta}\}^* \cup 
\{a^\eta,b^{\zeta}\}^* t  \{a^\zeta,b^{\eta}\}^* 
$
yields the conjugacy geodesic and shortlex minimal conjugacy languages 
(where the latter is computed using the ordering at the
start of this proof) as 
\begin{align*}
\geocl(G,Z) &=
\bigcup_{\eta,\zeta \in \{\pm 1\}} \{a^\eta,b^\zeta\}^*
\cup 
 \bigcup_{\eta \in \{\pm 1\}} \{a^\eta,b^{\eta}\}^* t\{a^\eta, b^{\eta}\}^*  \\
\mincl(G,Z) &= 
\bigcup_{\eta,\zeta \in \{\pm 1\}} (a^\eta)^*(b^\zeta)^*
\cup
 \bigcup_{\eta \in \{\pm 1\}} (a^\eta)^*(b^\eta)^* t.
\end{align*}
\end{proof}

The other generating set for $G=\Z^2  \rtimes \Z/2\Z$
considered by Cannon in~\cite{neushap} is
$
\{a,c,d,t\}^{\pm 1}
$
where $c=a^2$ and $d=ab$,
with presentation
$$
G=\langle a,c,d,t \mid ad=da, c=a^2, t^2=1,  
  tat = a^{-1}d \rangle.
$$
In the proof below we use these
generators, including a few more details for Cannon's
proof that $\geol$ is not regular.

We also require some more notation. Following \cite{kari}, we define the
insertion of $w$ into $z$, denoted by
$z\leftarrow w$, to be 
the set of all words of the form $z'wz''$, where
$z=z'z''$.
Given languages 
$L_1,L_2 \subseteq X^*$, we define also the insertion of $L_2$ into $L_1$,
\[
L_1 \leftarrow L_2 := \{ z\leftarrow w \mid z \in L_1, w \in L_2\} .
\]
Observe that, if $L_1$ and $L_2$ are regular, then
so is $L_1\leftarrow L_2$.

\begin{proposition}\label{prop:geovsgeocon}
Let
$G
\cong \Z^2 \rtimes \Z/2\Z$ be as defined above, with
generating set $X=\{a^{\pm 1},c^{\pm 1},d^{\pm 1},t\}$.
Then $\geocl(G,X)$ is regular, but $\geol(G,X)$ is not.
Moreover, for some orderings of the generating set,
$\sphcl(G,X)$ and $\mincl(G,X)$ are regular,
but $\sphl(G,X)$ is not.
\end{proposition}

\begin{proof}
The group $G=\Z^2 \rtimes \Z/2\Z$ can be embedded
in Euclidean 3-space $\R^3$ using 
the function  $f:G \rightarrow \R^3$ 
defined by
$f(g) := (i_g,j_g,\epsilon_g)$
where $a^{i_g}b^{j_g}t^{\epsilon_g}$ (with $b=tat$)
is the shortlex normal form over $Z$ for the element $g$ of $G$.
Multiplication by $a^{\pm 1}$, $c^{\pm 1}$, or $d^{\pm 1}$
has the effect of adding a horizontal
vector (parallel to the $z=0$ plane); 
in particular, for any $g \in G$ 
and $\mu \in \{1,-1\}$, we have
\begin{eqnarray*}
f(ga^\mu)&=&f(g)+\mu(1-\epsilon_g,\epsilon_g,0),\quad
f(gc^\mu)=f(g)+2\mu(1-\epsilon_g,\epsilon_g,0),\\
f(gd^\mu)&=&f(g)+\mu(1,1,0).  
\end{eqnarray*}

On the other hand, multiplication by $t$ adds
a vertical vector:  
\[ f(gt)=f(g)+(1-2\epsilon_g)(0,0,1).\]
Using the taxicab metric on $\R^3$ (defined by the norm
$||(x,y,z)|| = |x|+|y|+|z|$), we see that
the maximum distance between $f(g)$ and $f(gx)$ due to multiplication
by a single generator $x \in X$ is 2.  

For the remainder of the proof we order the generating
set $X$ by $a<c<A<C<d<D<t$,
where for ease of notation
we use a capital letter $A$ to denote the inverse
$a^{-1}$, etc.

First we show that the languages $\geol(G,X)$ and $\sphl(G,X)$
are not regular.
Consider the element $g := a^{2m}b^{2n}$ in $G$,
where $m,n \ge 1$, and let
$w$ be any word representing $g$. Then $w$ 
must contain an even (possibly zero) number of
occurrences of $t$, and, 
since the point $f(g)=(2m,2n,0)$ 
has taxicab metric distance $2m+2n$ 
from the identity of the group, and multiplication by $t$ does not affect the  
the first two coordinates, we see that
$w$ must involve 
at least $m+n$ letters from $\{a,c,d\}^{\pm 1}$.
If $m \ge n$, then the word 
$c^{m-n}d^{2n}$ of length $m+n$
is a geodesic representative for $g$,
and no geodesic representative can contain $t$.  
But if $m<n$ and $w$ does not contain $t$, then
$w$ must contain at least $2n$ occurrences of $d$
since the $y$-coordinate of $f(g)$ is $2n$;
then $c^{m-n}d^{2n}$ is a shortest representative of $g$ over
$\{a,c,d\}^{\pm 1}$.
However, the shorter word $c^mtc^nt$
also represents $g$, and
is geodesic. Moreover, any
word of length $m+n+2$ representing $g$
must contain two occurrences of $t$
and $m+n$ occurrences of letters among $c,d$,
with at most $m$ occurrences of $c$ before the first
$t$, and so $c^mtc^nt$ is the shortlex normal
form for $g$. 
Hence
\[
\geol(G,X) \cap c^*tc^*t = \sphl(G,X) \cap c^*tc^*t = \{c^mtc^nt \mid m < n\}.
\]
Since this last set is not regular, but
the set of regular languages is closed under intersection,
we conclude that neither
$\geol(G,X)$ nor $\sphl(G,X)$ is regular.

We shall now deduce the regularity of
$\geocl(G,X)$, $\mincl(G,X)$, and $\sphcl(G,X)$ from
the structure of the conjugacy classes
for this group discussed in the proof of Proposition~\ref{prop:sphclnotreg}.
(We note that $\sphcl(G,X)$ and
$\mincl(G,X)$ are also regular with various other orderings of the generators;
we chose the ordering in this proof 
because it results in the least complicated description.)

We begin by computing the portions of these three 
languages from conjugacy classes of the form $[a^mb^n]\subconj$ 
when $0 \le m,n$; without loss of generality we assume $m \le n$.
Using similar arguments to those above shows that geodesic
representatives of the group element $a^nb^m$ in this conjugacy class
have length $\lceil (m+n)/2 \rceil$ and have the form 
$(c^*d^*)^*$ for $m+n$ even and
$(c^*d^*)^*a(c^*d^*)^*$ for $m+n$ odd. 
However, the geodesic representatives of the 
other element $a^mb^n$ of this conjugacy class are strictly
longer; when $n-m \ge 3$ they all involve two
occurrences of $t$, and
geodesic representatives of $a^mb^{m+1}$ and $a^mb^{m+2}$
are given by $Ad^{m+1}$ and $Cd^{m+2}$, respectively.
Thus the only element of $[a^mb^n]\subconj$ with minimal length
up to conjugacy is $a^nb^m$. 
A similar argument shows that the only group element in 
$[A^mB^n]\subconj$ for $0 \le m \le n$ with minimal length
up to conjugacy is $A^nB^m$.
Thus when $m$ and $n$ are both nonnegative,
the portion of the set $\geocl(G,X)$
arising from the conjugacy classes
$[a^mb^n]\subconj$  and $[A^mB^n]\subconj$ is 
\[ (\{c,d\}^* \leftarrow \{ \1,a\}) \cup (\{C,D\}^* \leftarrow 
\{\1,A\})\]
and the portion of both $\mincl(G,X)$ and $\sphcl(G,X)$ is
\[ \{\1,a\}c^*d^* \cup \{\1,A\}C^*D^*.\]

It is more complicated to describe the
contributions to the three conjugacy
languages from $[a^mB^n]\subconj$ in the case that $0 < m,n$,
and we shall do this only briefly.
Note first that the sets of geodesic
representatives of $b^k$ for $k \ge 1$ are:
$\{Ad,dA\}$ for $k=1$; $\{Cd^2,dCd,d^2C,tct\}$ for $k=2$;
$\{tc^{k/2}t\}$ for $k \ge 4$ even; and $t(c^{(k-1)/2}\leftarrow a)t$ 
for $k \ge 3$ odd.

In the case that $3 \le m,n$, geodesics representing
both elements
$a^mB^n$ and $A^nb^m$ of this class require two
occurrences of the letter $t$, and both elements are
of minimal length up to conjugacy. 
The contribution to $\geocl(G,X)$ from these conjugacy classes is
the regular set
$L_1 \cup L_2 \cup L_3 \cup L_1^- \cup L_2^- \cup L_3^-$, where
\begin{eqnarray*}
L_1 &=& [c^+ \leftarrow \{a,c\}] \leftarrow [t (C^+\leftarrow \{A,C\})t],\\
L_2 &=& (c^+ \leftarrow t C^2C^* t) \leftarrow d ,\\
L_3 &=& ( c^2c^* \leftarrow tC^+ t) \leftarrow D.
\end{eqnarray*}
and $L_1^-,L_2^-,L_3^-$ are defined similarly to $L_1,L_2,L_3$,
but with all generators replaced by their inverses.
The words in $\mincl(G,X)$ are the shortlex normal
forms of all of the elements in the classes 
$[a^mB^n]\subconj$ with $3 \le m,n$, which we collect in
the regular expression
\[
\{c^2,ac\}c^* t \{C^2,AC\} C^* t
\cup \{C^2,AC\} C^* t \{c^2,ac\} c^* t. 
\]
Similarly, the intersection of this last set with
$\sphcl(G,X)$ 
is given by the regular expression
$
\{c^2 ,ac\}c^* t \{C^2,AC\}C^* t.
$

Suppose, on the other hand, that
either $m \in \{1,2\}$ and $n \ge 3$, or 
$m \ge 3$ and $n \in \{1,2\}$.
Then the conjugacy class
$[a^mB^n]\subconj$ contains a unique shortest group element, namely
$A^nb^m$ when $m \le 2$ or $a^mB^n$ when $n \le 2$.
The words of $\geocl(G,X)$
from classes of this type are in the regular set 
$L_4 \cup L_5 \cup L_4^- \cup L_5^-$, where
\begin{eqnarray*}
L_4 &=&  (c^2c^*\leftarrow D)\leftarrow \{\1,a\},\\
L_5 &=& ((c^2c^*\leftarrow D)\leftarrow D)\leftarrow \{a,c\}.
\end{eqnarray*}
and $L_4^-,L_5^-$ are similarly defined in terms of  $A,C,d$.
The portion of both $\mincl(G,X)$ and $\sphcl(G,X)$ from these classes is
given by the regular expression
\[
\{a,\1\}c^2c^*D \cup \{a,c\}c^2c^*D^2
 \cup \{A,\1\}C^2C^*d\cup \{A,C\}C^2C^*d^2.
\]
The finitely many classes $[a^mB^n]\subconj$
when both $m$ and $n$ are
at most $2$ give rise to finite subsets of the
three conjugacy languages.

Finally, we consider the conjugacy classes $[a^st]\subconj$
for $s \in \Z$.  
Recall that this class contains all
group elements of the form $a^mb^nt$ for $m,n \in \Z$ and $m+n=s$.
Suppose first that $s > 0$.
Geodesic representatives of these elements all contain
a single $t$. 
Geodesic representatives of $a^mb^nt$ have length
$\lceil s/2 \rceil + 1$ when $m,n\geq 0$, but are longer otherwise.
Hence $\geocl(G,X)$ consists of the
geodesic representatives of $a^mb^nt$ with $m,n \ge 0$.
The case in which $s<0$ is similar.
Thus the contribution to $\geocl(G,X)$ from classes
$[a^st]\subconj$ is 
\[
(\{c,d\}^* t \{c,d\}^*\leftarrow \{\1,a\}) \cup
(\{C,D\}^* t \{C,D\}^*\leftarrow \{\1,A\}).
\]
The subsets of $\mincl(G,X)$ and $\sphcl(G,X)$  arising from these
conjugacy classes are defined by regular
expressions
\[ \{\1,a\} c^* \{\1,d\} t c^* \cup t a c^*
 \cup \{\1,A\} C^* \{\1,D\} t C^* \cup t A C^*
\quad\mbox{and}\quad
\{\1,a\} c^* t \cup \{\1,A\} C^* t,\]
respectively.
\end{proof}

In the next two propositions we change the
virtually abelian group under consideration to
a semidirect product group 
$K = \Z^2 \rtimes D_8$.  
One of the generating transpositions of $D_8$  acts by
swapping the generators of $\Z^2$, and the other one
fixes the generators; that is, 
$$
K=\langle a,b,t,u \mid t^2=u^2=(tu)^4=1, ab=ba, a^t=b, a^u=a, a^u=b \rangle.
$$
The index 4 subgroup $\langle a,b,t \rangle$ of $K$ is
the group $G$ considered in Propositions~\ref{prop:sphclnotreg}
and~\ref{prop:geovsgeocon}.

\begin{proposition}\label{prop:z2d8good}
Let
$K \cong \Z^2 \rtimes D_8$ be as defined above,
where the generators $Z'=\{a^{\pm 1},b^{\pm 1},t,u\}$
are ordered with $a^{\pm 1} < b^{\pm 1} < t < u$.
Then $\geol(K,Z')$, $\sphl(K,Z')$,  $\geocl(K,Z')$, 
and $\mincl(K,Z')$ are regular languages, but $\sphcl(G,Z')$ is not.
\end{proposition}

\begin{proof}
The shortlex language is
\[
\sphl(K,Z') = \{a^ib^jv \mid i,j \in \Z, v \in \{1,t,u,tu,ut,tut,utu,tutu\}\}
\]
and the geodesic language $\geol(K,Z')$ is $\bigcup_{\eta,\zeta \in \{\pm 1\}}
(L_1^{\eta,\zeta} \cup L_3^\eta \cup L_2^{\eta,\zeta})$, where
\begin{eqnarray*}
L_1^{\eta,\zeta} &=& \{a^\eta,b^\zeta\}^*\leftarrow \{\1,u\},\\
L_2^{\eta,\zeta} &=& (\{a^\eta,b^\zeta\}^*\leftarrow\{\1,u\})\leftarrow
        t (\{a^\zeta,b^\eta\}^* \leftarrow u) t,\\
L_3^{\eta,\zeta} &=&\{a^\eta,b^\zeta\}^* t
              \{a^\zeta,b^\eta\}^*\leftarrow \{\1,u\}.
\end{eqnarray*}

By successively conjugating shortlex normal forms by generators, we find
that the conjugacy classes are given by 

\begin{eqnarray*}
[a^ib^j]\subconj&=&\{a^ib^j,a^jb^i\}, \quad
[a^ib^j t]\subconj = 
[a^ib^j utu]\subconj = 
         \{a^{i+k}b^{j-k}v \mid k \in \mathbb{Z}, v \in \{t,utu\} \}, \\\  
[a^ib^j u]\subconj&=&
[a^jb^i tut]\subconj=
          \{a^ib^ju,a^jb^itut\}, \quad
[a^ib^j tutu]\subconj=\{a^ib^j tutu,a^jb^i tutu\}, \\\ 
[a^ib^j tu]\subconj &=& 
[a^ib^j ut]\subconj =
          \{a^{i+k}b^{j-k}v \mid k \in \mathbb{Z}, v \in \{tu,ut\} \}.  \\
\end{eqnarray*}
As a consequence, 
\begin{eqnarray*}
\geocl(K,Z')&=&
\bigcup_{\eta,\zeta \in \{\pm 1\}} (L_1^{\eta,\zeta} \cup
    {L'}_2^{\eta,\zeta}) \cup
     (\bigcup_{\eta \in \{\pm 1 \}} {L'}_3^{\eta}),
\quad\mbox{where}\\
{L'}_2^{\eta,\zeta} &=&(\{a^\eta,b^\zeta\}^*\leftarrow u)\leftarrow
             t (\{a^\zeta,b^\eta\}^*\leftarrow u) t,\\
{L'}_3^{\eta} &=& (\{a^\eta,b^\eta\}^*\leftarrow t)\leftarrow\{1,u\},
\end{eqnarray*}
and 
\[
\mincl(K,Z') =
\bigcup_{\eta,\zeta \in \{\pm 1\}} 
   (a^\eta)^*(b^\zeta)^*\{1,u,tutu\}
\cup
\bigcup_{\eta \in \{\pm 1\}} 
  \{a^\eta)^*(b^\eta)^* \{t,tu,ut\}.
\]
However, just as for the finite index subgroup
$G$ of $K$ in Proposition~\ref{prop:sphclnotreg},
the intersection $\sphcl(K,Z') \cap a^*b^* = 
\{a^ib^j \mid i>j\}$ is not regular, and so
this \wsphcl\ is not regular.
\end{proof}

\begin{proposition}\label{prop:vab_geoconnotreg}
Let
\begin{eqnarray*}
K
&=&
  \langle a,c,d,t,u \mid\\
&& \ \ \ \ t^2=u^2=(tu)^4=1, ad=da, c=a^2, tat=a^{-1}d, uau=a, udu=d \rangle\\
&\cong&\Z^2 \rtimes D_8,\end{eqnarray*}
with generating set $X'=\{a^{\pm 1},c^{\pm 1},d^{\pm 1},t,u\}$.
Then $\geol(K,X')$ and $\geocl(K,X')$  are not regular.
Moreover, if the set $X'$ is ordered by
$a<c<a^{-1}<c^{-1}<d<d^{-1}<t<u$, then none of the languages
$\sphl(K,X')$, $\mincl(K,X')$, and $\sphcl(K,X')$ is
regular.
\end{proposition}

\begin{proof}
In order to analyze 
geodesic representatives for elements of $K$ over $X'$, we
use similar arguments to those in the proof of
Proposition~\ref{prop:geovsgeocon}, in that we use an embedding 
$f:K \rightarrow \R^2 \times D_8$
with the map $f(g) := (i_g,j_g,h_g)$, where $g = a^{i_g} b^{j_g} h_g$
with $b=a^t$ and $h_g \in \langle t,u \rangle \cong D_8$. 
Geodesic representatives over $X'$ of the elements of
the subgroup $\Z^2 \rtimes \Z/2\Z = \langle X \rangle$ of $K$
where $X=\{a^{\pm 1},c^{\pm 1},d^{\pm 1},t\}$
cannot contain an occurrence of the letter $u$,
and so the proof
of Proposition~\ref{prop:geovsgeocon} also shows
that $\geol(K,X')$ and $\sphl(K,X')$ are not regular.

To show further that $\geocl(K,X')$ and $\mincl(K,X')$
are not regular, 
consider the word $w_{mn}:=c^mtc^ntu$ for $m,n \ge 0$.

When $n \le m$, then
(as in the proof of
Proposition~\ref{prop:geovsgeocon})
$w_{mn} =_K c^{m-n}d^{2n}u$, and 
so $w_{mn} \not\in \geol(K,X')$.

When $m < n$,
we claim that $w_{mn} \in \mincl(K,X')$; 
this will imply that the intersections of each of
$\geocl(K,X')$ and $\mincl(K,X')$ with the regular
language $c^*tc^*tu$ are not regular; hence neither 
conjugacy language is regular.
  
The conjugacy class of $w_{mn}$ is $[w_{mn}]\subconj=[a^{2m}b^{2n}u]\subconj =
\{a^{2m}b^{2n}u,a^{2n}b^{2m}tut\}$, from the proof of 
Proposition~\ref{prop:z2d8good}.  
Using the embedding above, we have
$f(a^{2n}b^{2m}tut) = (2n,2m,tut)$.  Since we need at least
$m+n $ letters in $\{a,c,d\}^{\pm 1}$ to reach this point, and
$tut$ is a geodesic word in $\langle t,u \rangle \cong D_8$, 
every geodesic representative of $a^{2n}b^{2m}tut$ has length at least $m+n+3$.  

On the other hand, $f(w_{mn}) = f(a^{2m}b^{2n}u) = (2m,2n,u)$. 
As in the proof of Proposition~\ref{prop:geovsgeocon}, 
in order to reach a point in
$\R^2 \times D_8$ with first two coordinates $(2m,2n)$ we need at least
$m+n$ occurrences of letters in $\{a,c,d\}^{\pm 1}$.
Suppose for a contradiction that $v$ is a word with $v =_K w_{mn}$
and $l(v) < l(w_{mn})$. So $l(v) \le m+n+2$ and $v$ involves at least
$m+n$ letters in $\{a,c,d\}^{\pm 1}$. Since $v$ is not in the normal subgroup
$\langle a,c,d,t,(tu)^2 \rangle$ of index 2 in $K$, the number of
occurrences of $u$ in $v$ must be odd, and hence must be exactly 1.
We see similarly that the number of occurrences of $t$ is even, so it must
be 0. So deleting the single occurrence of $u$ from $v$ would give a word over
$\{a,c,d\}^{\pm 1}$ of length at most $m+n+1$ 
for $a^{2m}b^{2n}$ which, as we saw in 
Proposition~\ref{prop:geovsgeocon}, does not exist when $m<n$.
Thus the element $a^{2m}b^{2n}u$ of $[w_{mn}]\subconj$ is of
minimal length up to conjugacy, and the word
$w_{mn}$ representing this element lies in $\geocl(K,X')$.
All geodesic representatives of this element lie
in $\{c, d, t\}^* \leftarrow u$, and among these the word
with the longest initial prefix in $c^*$ is $w_{mn}$.
Therefore $w_{mn} \in \mincl(K,X')$ as well, finishing the claim.

In order to prove that $\sphcl(K,X')$ is not regular,
we consider the word $v_{mn}:=c^mtc^nutu$ for $m,n \ge 0$.
The corresponding conjugacy class from the proof of
Proposition~\ref{prop:z2d8good}
is $[v_{mn}]\subconj =
\{a^{2m}b^{2n}tutu,a^{2n}b^{2m}tutu\}$.  The words
$c^mtc^nutu$ and $c^ntc^mutu$ are the shortlex least
words representing the elements $a^{2m}b^{2n}tutu$ and
$a^{2n}b^{2m}tutu$, respectively.  Thus the element of
$\sphcl(K,X')$ corresponding to this conjugacy class is
the representative with the longest initial string of 
the letter $c$.  That is, 
\[
\sphcl(K,X') \cap c^*tc^*utu = \{c^mtc^nutu \mid m \ge n\},
\] 
which is not regular.
\end{proof}

\section*{Acknowledgments}

The first two authors were partially supported by the Marie Curie Reintegration Grant 230889, and by scheme 2 grants from the London Mathematical Society. 
The first named author was also supported by the Swiss National Science 
Foundation grants Ambizione PZ00P-136897/1 and Professorship FN PP00P2-144681/1.  
The second author also acknowledges partial support by grants from the Simons
Foundation (\#245625) and the National Science Foundation (DMS-1313559).


\bigskip

\textsc{L. Ciobanu,
Mathematics Department,
University of Neuch\^atel,
Rue Emile-Argand 11,
CH-2000 Neuch\^atel, Switzerland
}

\emph{E-mail address}{:\;\;}\texttt{laura.ciobanu@unine.ch}
\medskip

\bigskip

\textsc{S. Hermiller,
Department of Mathematics,
University of Nebraska,
Lincoln, NE 68588-0130, USA 
}

\emph{E-mail address}{:\;\;}\texttt{smh@math.unl.edu}
\medskip

\bigskip

\textsc{D. Holt,
Mathematics Institute,
Zeeman Building,
University of Warwick,
Coventry CV4 7AL, UK
}

\emph{E-mail address}{:\;\;}\texttt{D.F.Holt@warwick.ac.uk}
\medskip

\bigskip

\textsc{S. Rees,
Department of Mathematics,
University of Newcastle,
Newcastle NE1 7RU,
UK
}

\emph{E-mail address}{:\;\;}\texttt{Sarah.Rees@ncl.ac.uk}
\medskip

\end{document}